\documentclass[11pt]{amsart}

\usepackage[utf8]{inputenc}
\usepackage[margin=3.1cm]{geometry}
\usepackage{xcolor}
\usepackage{graphicx}
\usepackage{todonotes}
\usepackage{subcaption}
\usepackage{mathtools}

\usepackage{mathtools}

\mathtoolsset{showonlyrefs}

\mathtoolsset{showonlyrefs}

\usepackage{amsmath, amssymb, amsthm}
\usepackage{esint}
\usepackage{verbatim}
\usepackage{dsfont}
\usepackage{enumitem}
\usepackage{helvet}

\DeclareMathOperator{\dist}{dist}

\DeclareMathOperator{\conv}{conv}

\DeclareMathOperator{\cof}{cof}
\newcommand{\Hm}{\mathcal{H}} 
\newcommand{\N}{\mathbb{N}} 
\newcommand{\R}{\mathbb{R}} 
\newcommand{\C}{\mathbb{C}} 
\newcommand{\1}{\mathds{1}} 
\newcommand{\calA}{\mathcal{A}}
\newcommand{\id}{\mathrm{id}}

\newcommand{\Skew}{\mathrm{Skew}}
\newcommand{\Symm}{\mathrm{Symm}}

\newcommand{\GSBD}{\mathrm{GSBD}}
\newcommand{\GSBV}{\mathrm{GSBV}}
\newcommand{\SBD}{\mathrm{SBD}}
\newcommand{\SBV}{\mathrm{SBV}}
\newcommand{\BD}{\mathrm{BD}}

\newcommand{\aplim}{\mathrm{ap\,lim}}
\newtheorem{thm}{Theorem}[section]
\newtheorem{prop}[thm]{Proposition}

\newtheorem{lemma}[thm]{Lemma}
\newtheorem{rem}[thm]{Remark}
\newtheorem{defi}[thm]{Definition}
\newcommand{\JG}{\textcolor{black}}
\newcommand{\Jedit}{\textcolor{black}}

\title{The Euler-Bernoulli limit of thin brittle linearized elastic beams}
\date{}

\begin{document}

\author[J. Ginster]{Janusz Ginster}
\address{Janusz Ginster\\Institut f\"ur Mathematik \\Humboldt-Universit\"at zu Berlin\\Unter den Linden 6\\ 10099 Berlin, Germany}
\email{janusz.ginster@math.hu-berlin.de}

\author[P. Gladbach]{Peter Gladbach}
\address{Peter Gladbach\\ Insitut f\"ur Angewandte Mathematik \\ Universit\"at Bonn \\ Endenicher Allee 60 \\ 53115 Bonn, Germany}
\email{gladbach@iam.uni-bonn.de}


\begin{abstract}
We show that the linear brittle Griffith energy on a thin rectangle $\Gamma$-converges after rescaling to the linear one-dimensional brittle Euler-Bernoulli beam energy.

In contrast to the existing literature, we prove a corresponding sharp compactness result, namely a suitable weak convergence after subtraction of piecewise rigid motions with the number of jumps bounded by the energy.
\end{abstract}

\maketitle

\tableofcontents
\section{Introduction}

We consider a model of brittle linearly elastic Euler-Bernoulli beams of fixed length $L>0$ and variable thickness $h \searrow 0$. We define the undeformed beam as
\begin{equation}
\Omega_h := (0,L) \times (-h/2,h/2) \subset \R^2.
\end{equation}

We fix throughout the paper a positive definite Cauchy stress tensor $\C\in \R^{2\times 2 \times 2 \times 2}$ such that
\begin{equation}\label{eq: tensor properties}
\C_{ijkl} = \C_{jikl} = \C_{klij}\quad\text{ and }F:\C F \geq c|F+F^T|^2
\end{equation}
for some $c>0$.

This allows us to define the energy of a displacement field $w\in \GSBD^2(\Omega_h)$ as
\begin{equation}\label{def: Fh}
F_h(w) := \frac1{h^3}\int_{\Omega_h} \frac12 e w : \C e w \,dx + \frac\beta{h} \Hm^1 (J_w),
\end{equation}
\Jedit{where $\Hm^1$ denotes the $1$-dimensional Hausdorff measure and $J_w$ is the jump set of the function $w$, see Section \ref{sec: GSBD} (also for the definition of the space $\GSBD^2(\Omega_h)$ and its properties)}.

The elastic prefactor $\frac1{h^3}$ is chosen so that no stretching occurs and we recover the bending regime in the limit. The fracture prefactor $\frac{\beta}{h}$ denotes the material toughness, which has to scale as $\frac1h$ to recover the number of fracture points in the limit.

We note that by the symmetry of $\C$, for $w\in \GSBV(\Omega_1)$ we have $ew : \C ew = \nabla w : \C \nabla w$ almost everywhere, where $ew = (\nabla w + \nabla w ^T)/2$. For the more general definition of $ew$ and the space $\GSBD^2(\Omega_h)$, see Section \ref{sec: GSBD}.

Because we let $h\to 0$, we perform the usual change of variables: Let $y \in L^2(\Omega;\R^2)$ and $h>0$. Define $w \in L^2(\Omega_h)$ by $w(x_1,x_2) = y(x_1,x_2 / h)$ and correspondingly define the energy $E_h: L^0(\Omega_1) \to [0,\infty]$ by
\begin{equation}\label{eq: def energy}
E_h(y) = \begin{cases}
F_h(w) &\text{ if } w \in \GSBD^2(\Omega_h), \\
+ \infty &\text{ else,}
\end{cases}
\end{equation}
where $F_h$ is defined in \eqref{def: Fh}.
Formally, for regular functions $y$ (e.g. $y \in \SBV^2(\Omega_1;\R^2)$) by a change of variables it holds
\begin{equation}\label{eq: change of var}
E_h(y) = \frac{1}{h^2} \int_{\Omega_1} \frac12 (\partial_1 y, \frac1h \partial_2 y) : \C (\partial_1 y, \frac1h \partial_2 y) \,dx + \beta \int_{Jy} |(\nu_1,\frac1h\nu_2)|\,d\Hm^1,
\end{equation}
where $\nu\in S^1$ is the measure-theoretic normal to the jump set $\JG{J_y}$. 
In general by the definition it is not true that $F_h(y) < \infty$ implies that $\nabla y$ exists as a function in $L^2(\Omega_1;\R^{2\times 2})$. 
We will use Korn's inequality for functions with a small jump set (see \cite{CaChSc20}) to establish the relation \eqref{eq: change of var} on a large set.

We show that the sequence of energies $E_h$ $\Gamma$-converges to a limit energy $E_0$ which is only finite on the following set of admissible limit configurations
\begin{equation}
\calA := \{y\in \SBV^2(\Omega_1;\R^2)\,:\,D_2 y = 0, \partial_1 y_1 = 0, \partial_1 y_2\in \SBV(\Omega_1;\R)\}.
\end{equation}

In other words, $y$ is a function only of $x_1$, $y_1$ is piecewise constant on $(0,L)$, and $y_2$ is piecewise $W^{2,2}$ on $(0,L)$ with both $y_2$ and $\partial_1 y_2$ jumping finitely many times on $(0,L)$. The limit energy $E_0:L^1(\Omega_1;\R^2)\to [0,\infty]$ is given by (see Theorem \ref{thm: Gamma limit})
\begin{equation}\label{eq: limit}
E_0(y) := \begin{cases}
\frac1{24} \int_{\Omega_1} a|\partial_1\partial_1 y_2|^2\,dx + \beta \Hm^1(J_y \cup J_{\partial_1 y})&,\text{ if }y\in\calA\\
\infty&,\text{ otherwise}.
\end{cases}
\end{equation}

We note that $E_0$ is in fact a one-dimensional energy
\begin{equation}
E_0(y) = \int_0^L \frac{a}{\Jedit{24}} |y_2''|^2\,dx_1 + \beta \#(J_{y_2} \cup J_{y_2'} \cup J_{y_1})
\end{equation}
for $y(x_1,x_2) = (y_1(x_1),y_2(x_1))\in \calA$, with $y_1:(0,L)\to \R$ piecewise constant, which for $y_1=0$ coincides with the one-dimensional version of the Blake-Zisserman model for image denoising, \cite{BlakeZisserman87} (see also \cite{BlakeZisserman2,BlakeZisserman3,BlakeZisserman4} for its analysis).

The bending constant $a>0$ is defined as usual in Euler-Bernoulli beam theory as
\begin{equation}\label{eq: bending constant}
a := \inf_{b,c\in \R} \begin{pmatrix} 1 & b\\ 0&c \end{pmatrix} :\C \begin{pmatrix} 1 & b\\ 0&c \end{pmatrix}.
\end{equation}
The vector $(b,c)^T$ can be seen as an optimal shear response to a unit curvature. We note that unlike in Euler-Bernoulli beam theory, more complex models such as Ehrenfest-Timoshenko beam theory keep track of the additional shear variable in addition to the displacement $y$, leading to generally higher energy.

As is usual in the considered scaling regime in dimension reduction, the limit energy penalizes bending moments, which are not penalized in $E_h$. The emergence of a bending energy can be seen heuristically by taking
\[
y_h(x_1,x_2) := y(x_1,0) - x_2 h \nabla y_2(x_1,0) - \frac12 x_2^2 h^2 \partial_1\partial_1 y(x_1,0) \begin{pmatrix} b \\ c \end{pmatrix},
\]
where we need to subtract $x_2 h \nabla y_2(x_1,0)$ from $y$ so that the symmetric part of the matrix $(\partial_1 y_h, \frac1h \partial_2 y_h)$ converges to $0$. The precise calculation is found in Section \ref{sec: upper}. 

Similar $\Gamma$-convergence results have already been proven in the $n$-dimensional setting, see \cite{AlTa20,BaHe16}. However, here we show a stronger complementing compactness theorem, see also the discussion in Section \ref{sec: discussion}.
The complementing compactness result can be illustrated as follows.
Already without the possibility of fracture it is clear that sequences of functions with a bounded elastic energy are not precompact in a reasonable way as the elastic energy is invariant under the addition of rigid motions which form a non-compact set. 
Using Korn's inequality, in this setting it can be expected that one can identify a sequence of rigid motions $A_h x + b_h$, $A_h \in \Skew(2), b_h \in \R^2$ such that the difference of $w_h$ and the rigid motions is precompact after being rescaled to $\Omega_1$.
Additionally, fracture can occur and different rigid motions might be present on different parts of $\Omega_h$ which have been broken apart from one another. 
However, the form of the energy $E_h$ suggests that the only way to break apart larger parts of $\Omega_h$ is along essentially vertical cracks.
Hence, a reasonable compactness result needs to identify the different parts of $\Omega_h$ which have been broken apart from one another along vertical lines together with the corresponding dominant rigid motions, and
additionally detect the asymptotically vanishing part of $\Omega_h$ that is disconnected from the rest of $\Omega_h$ along non-vertical lines. 
In fact, we show that for an energy-bounded sequence $y_h$ there are $x_2$-independent, piecewise-constant functions $A_h$ and $b_h$  and asymptotically vanishing sets $\omega_h$ such that the sequence $\1_{\Omega_h \setminus \omega_h} (y_h - A_h (x_1,hx_2)^T - b_h)$ is precompact in $L^2(\Omega_1;\R^2)$.
Moreover, the functions $A_h$ and $b_h$ are constructed carefully enough so that the modified sequence does not have asymptotically more jump than $y_h$ which is important for meaningful asymptotic lower bounds, see Theorem \ref{thm: compactness}. A key tool in this analysis will be a Korn's inequality for $\GSBD^2(\Omega)$, see \cite{CaChSc20}.\\

Next, we present a brief overview over existing results in the literature.

\section{Elasticity, beams, and fracture}

\subsection{Geometric and linearized elasticity}

We provide in this section a brief overview over the relevant theories. First we start with unfractured homogeneous hyperelastic materials. Here a stress-free reference configuration $\Omega \subset \R^d$ undergoes a deformation $u:\Omega \to \R^d$. The geometric hyperelastic energy is then given by
\[
\int_\Omega W(\nabla u)\,dx,
\]
where $W:\R^{d\times d} \to [0,\infty)$ denotes the $C^2(\R^{d\times d})$ hyperelastic energy density. We make the physical assumptions that $W(\id) = 0$, i.e. $u(x) = x$ has the lowest possible energy, and $W(RA) = W(A)$ for $R\in SO(d)$, i.e. rotations have no effect on the energy. The most-studied energy densities are those with quadratic growth at $SO(d)$ and at $\infty$, where $\dist^2(A,SO(d)) \lesssim W(A) \lesssim \dist^2(A,SO(d))$. A central result in the theory of hyperelastic materials is the geometric rigidity result by Friesecke, James, M\"uller \cite{FJM02}, which states that for open connected Lipschitz domains $\Omega \subset \R^d$, there exists a constant $C(\Omega)>0$ such that
\begin{equation}\label{eq: rigidity}
\min_{R\in SO(d)} \int_{\Omega} |\nabla u - R|^2\,dx \leq C(\Omega) \int_\Omega \dist^2(\nabla u, SO(d))\,dx.
\end{equation}

In particular, we have that whenever $\int_\Omega W(\nabla u_k)\,dx\to 0$, up to subsequences and fixed rotations $R_k\in SO(d)$ and shifts $b_k\in \R^d$, we have $R_k^T (u_k(x) - b_k) \rightharpoonup x$ weakly in $H^1(\Omega;\R^d)$. For deformations with small hyperelastic energy, we may thus write $R_k^T(u_k(x) - b_k) = x + w_k(x)$, with $w_k(x)$ converging weakly to zero in $H^1(\Omega;\R^d)$. A Taylor expansion of the energy yields
\begin{equation}
\int_{\Omega} W(\nabla u_k)\,dx \approx \frac12 \int_{\Omega} \nabla w_k(x):\C \nabla w_k(x)\,dx,
\end{equation}
where $\C = D^2W(\id)\in \R^{d\times d\times d \times d}$. The quadratic growth conditions on $W$ and Schwarz's theorem then guarantee \eqref{eq: tensor properties}.
\Jedit{For a rigorous derivation via $\Gamma$-convergence, see \cite{DMPe02}.}

The dynamics of the resulting quadratic form dealing with infinitesimal displacements $|w| \ll 1$ are commonly referred to as linearized elasticity, and form an important part of the physics and engineering literature, see e.g. \cite{gurtin1973linear}. In particular, they are often times simpler to deal with than the geometrically nonlinear version.

For example, applying \eqref{eq: rigidity} to small deformations yields Korn's inequality
\begin{equation}
\min_{A\in \Skew(d)}\int_\Omega |\nabla u - A|^2 \,dx \leq C(\Omega) \int_\Omega |\nabla u + \nabla u^T|^2 \,dx,
\end{equation}
which can be proved using elementary methods and was in fact proved by Korn in \cite{korn1909einige}.

\subsection{Thin elastic structures}

In contrast to full bodies, lower dimensional structures have potentially lots of isometric embeddings into $\R^d$. A famous example is the Nash-Kuiper theorem \cite{nash1954c1}, which states that for every Riemannian $m$-manifold $M$ and every smooth $1$-Lipschitz map $f:M\to \R^d$ with $d> m$, there is an isometric $C^1$ immersion of $M$ into $\R^d$ that is arbitrarily close in $L^\infty(M)$ to $f$.

Compare that to open sets $\Omega \subset \R^d$, where every isometric $C^1$ deformation $w:\Omega \to \R^d$ must be a rigid motion by \eqref{eq: rigidity}.

Thin structures are slightly thickened versions of submanifolds. The simplest nontrivial example is the Euler-Bernoulli beam $\Omega_h := (0,L) \times (-h/2,h/2) \subset \R^2$. In the geometrically nonlinear setting, a deformation carries low hyperelastic energy if the midsection $(0,L) \times \{0\}$ is isometrically embedded, i.e. $|\partial_1 u(x_1,0)| = 1$. In that case,
\begin{equation}\label{eq: nonlinear bending}
\int_{\Omega_h} W(\nabla u)\,dx \approx h^3\int_0^L \frac{a}{24}|\partial_1\partial_1 u(x_1,0)|^2\,dx_1, \text{ if }|\partial_1 u(x_1,0)| = 1,
\end{equation}
where $a>0$ is defined by \eqref{eq: bending constant}, and $\C = D^2W(\id)$.

In contrast, starting with the linearized elastic energy,
we find that
\begin{equation}
\int_{\Omega_h} \frac12 ew:\C ew \,dx \approx h^3 \int_0^L \frac{a}{24}|\partial_1 \partial_1 w_2(x_1,0)|^2\,dx_1, \text{ if } \partial_1 w_1(x_1,0) = 0
\end{equation}
which is the linearized version of the bending energy \eqref{eq: nonlinear bending}, c.f.~\cite{Ciarlet}. The resulting one-dimensional energy is named the Euler-Bernoulli energy after its originators. See e.g.~\cite{gurtin1973linear, Antman2005} for further reading.

We note that in the scaling regime $\int_{\Omega_h} W(\nabla u)\,dx \approx h$, stretching is possible and dominates the energy over bending. The theory is generally called string theory, see e.g. \cite{Antman2005}. For its two-dimensional analogue, the so-called membrane theory, see, for example, \cite{LDR95,LDR96,Haf08,Haf06}.

Generalizing from thin structures in the plane to thin structures in three-dimensional space, we differentiate between beams or rods of the type $R_h := (0,L) \times hS \subset \R^3$, with $S\subset \R^2$ open, bounded, connected, and plates $P_h:= (0,L)\times (0,L) \times (-h/2,h/2)$. Both, linear and nonlinear variational models exist for both, see e.g. \cite{MoMu04} for beams and \cite{FJM02, FJM06} for plates. We note also that shells, which are curved analogues of plates, have been similarly studied, see e.g. \cite{friesecke2003derivation}.

\subsection{Griffith's model of fracture}

Fracture is one of multiple failure modes in elastic structures. Fracture occurs along codimension-one hypersurfaces called cracks, where the deformation is discontinuous. We differentiate between cohesive fracture, where the energy depends on the magnitude of the discontinuity, and brittle fracture, which we discuss in this article, where the total energy depends only on the surface measure of the crack. 

For an open reference configuration $\Omega\subset \R^d$ and a displacement field $w:\Omega \to\R^d$ which is $C^1$ outside a closed rectifiable hypersurface $\Gamma\subset \Omega$, we define the Griffith brittle fracture energy (see \cite{FrMa98,Gr21}) as
\begin{equation}
E(w) := \inf\left\{\int_{\Omega \setminus \Gamma} \frac12 ew:\C ew\,dx + \beta \Hm^{d-1}(\Gamma)\,:\,\Gamma \subset \Omega\textnormal{ closed} \text{ s.t. } w\in C^1(\Omega\setminus \Gamma) \right\}.
\end{equation}

Here $\beta>0$ is the material toughness, i.e. the surface tension of \JG{the} crack surface. Expectedly, the space of piecewise $C^1$ deformations generally does not contain the minimizers of $E$, which led to the characterization of the energy space for $E$ in \cite{dal2013generalised}, the space of generalized functions of bounded deformation $\GSBD^2(\Omega)$, whose definition and key properties we recount in Section \ref{sec: GSBD}.

The study of fracture in thin materials has seen advancement in recent years. 
In the nonlinear setting in \cite{FoBr01} the authors study the derivation of a membrane theory in which stretching is dominant, see also \cite{BoFoLeMa}. Recently, Schmidt showed in \cite{Sc17} that the nonlinear version of $E_h$ $\Gamma$-converges to
\begin{equation}
\int_0^L \frac{a}{24} |y''(x_1)|^2\, dx_1 + \beta \#(J_y \cup J_{\partial_1 y}),\text{ if }y\in \SBV((0,L);\R^2), |y'(x_1)| = 1 \text{ a.e.,}
\end{equation}
which is the nonlinear analogue to the limit energy $E_0$. 
In \cite{BaHe16} and \cite{AlTa20} \JG{the authors} study the asymptotics of an $n$-dimensional analogue of the energy $E_h$, see also \cite{Ba,AlBeMiPe21,BaBaBoHeMa14} for the antiplane setting. Using a slightly different rescaling of the function $y_h$, c.f.~\cite{Ciarlet}, the authors obtain the limiting energy
\[
\int_{\Omega_1} Q(\nabla y) \, dx + \beta \Hm^{d-1}(J_y),
\]
where $Q$ is a quadratic form and $y$ is of the form $y_i(x_1,\dots,x_d) = \bar{y}_i(x_1,\dots,x_d) - x_d \partial_i y_d(x_1,\dots,x_d)$ for $i=1,\dots, d-1$ and $y_d$ does not depend on $x_d$. 
Although very similar to the result presented here, we note that the used techniques are rather different. In order to identify the specific form of the limiting $y$ in \cite{BaHe16} the authors study the distributional symmetric gradient of $y$ together with convolution techniques, in \cite{AlTa20} the authors use a delicate approximation argument in $\GSBD^2$. 
In contrast our proofs are based on rigidity arguments which are much closer to the techniques used in \cite{Sc17}, see also \cite{FJM02,MoMu04}.
This allows to obtain more control on the rescaled gradient \Jedit{$\nabla_h y_h = (\partial_1 y_h, \frac1h \partial_2 y_h, \frac1h \partial_3 y_h)$}.
In the presented setting this enables us to obtain an improved compactness statement and a short proof for the identification of the limiting configurations. Moreover, in other problems the additional control of $\nabla_h y_h$ is crucial. For example in the derivation of a rod theory the information about torsion is stored in the limit of $\nabla_h y_h$ and cannot be seen in the limiting $y$, see \cite{GiGl21}.

We note that our result deals with the slightly simpler linear energy but uses different methods, which can be generalized to the linear theory in higher dimensions.

\section{Notation}

Throughout the paper $C>0$ is a generic constant that may change from line to line.
Moreover, we use standard notation $x = (x_1,\dots,x_d)$ for vectors in $\R^d$. In particular, we will identify vectors with its transpose wherever it simplifies notation.
At several points, the components of vector-valued functions with a subscript, e.g. $w_h:\R^2\to\R^2$, are denoted $(w_h)_1, (w_h)_2$.
We say that two vectors $v,w\in \R^d$ are parallel, $v\|w$, if they are linearly dependent.
The space of symmetric and skew-symmetric $\R^{d\times d}$ matrices will be denoted by $\Symm(d)$ and $\Skew(d)$, respectively.
We use standard notation for the Lebesgue measure $\mathcal{L}^d$ and the $s$-dimensional Hausdorff measure $\mathcal{H}^s$. Moreover, for a Lebesgue-measurable set $B \subseteq \R^d$ we write $|B|$ for $\mathcal{L}^d(B)$.
Moreover, we use standard notation for Lebesgue and Sobolev spaces $L^p$ and $W^{1,p}$.
Lastly, for a set of finite perimeter $E \subseteq \R^d$ we write $\partial^* E$ for its reduced boundary, cf.\cite{AmbrosioFuscoPallara}.

\section{Main results}

We now state our main results.
We start with the compactness result.

\begin{thm}\label{thm: compactness}
Let $E_h$ be defined as in \eqref{eq: def energy} and $(y_h)_h \subseteq L^2(\Omega_1;\R^2)$ such that it holds $\sup_{h > 0} E_h(y_h) < \infty$. Then there is a subsequence (not relabeled), a function $y\in \calA$, a sequence of sets $\sigma_h \subset \Omega_1$ of finite perimeter with measure theoretic normal $\nu \in S^1$, and sequences of piecewise constant functions $\overline{A}_h:(0,L) \to \Skew(2)$, $\overline{b}_h:(0,L) \to \R^2$ such that
\begin{enumerate}[label=(\roman*)]

\item 
\begin{equation}
\lim_{h\to 0} \int_{\Omega_1\setminus \sigma_h} |y_h(x) - \overline{A}_h(x_1)(x_1,hx_2)^T - \overline{b}_h(x_1) - y(x)|^2\,dx  = 0 .
\end{equation} \label{item: convergence y_h}

\item  $|\sigma_h| \to 0$ and
\begin{equation}
\sup_{h>0} \int_{\partial^*\sigma_h} |(\nu_1,\frac1h \nu_2)|\,d\Hm^1 < \infty.
\end{equation} \label{item: bounds baromegah} 

\item  We have
\begin{equation}
\Hm^1(J_y \cup J_{\partial_1 y}) \leq  \left\lfloor\liminf_{h\to 0} \int_{J_{y_h}} |(\nu_1,\frac1h \nu_2)| \,d\Hm^1 \right\rfloor
\end{equation}
and, for every fixed $\bar{h}>0$,
\begin{equation}
 \#(J_{\overline{A}_{\bar h}} \cup J_{\overline{b}_{\bar h}}) \leq  \left\lfloor\liminf_{h\to 0} \int_{J_{y_h}} |(\nu_1,\frac1h \nu_2)| \,d\Hm^1 \right\rfloor.
\end{equation} \label{item: inequality jumpsets}
Here $\lfloor \cdot \rfloor:[0,\infty)\to \N$ is the floor function $x\mapsto \max\{n\in\N\,:\,n\leq x\}$.
\end{enumerate}
\end{thm}

\begin{rem}\label{rem: incomplete}
In other words, $\overline{A}_h$ and $\overline{b}_h$ only jump where the limit jump density of $y_h$ is at least one. We emphasize that one may not replace (iii) by the better estimate
\begin{equation}
\#(J_{\overline{A}_h} \cup J_{\overline{b}_h}) \leq  \left\lfloor \int_{J_{y_h}} |(\nu_1,\frac1h \nu_2)| \,d\Hm^1 \right\rfloor.
\end{equation}

To see this, consider the sequence of triangles in $\overline{\Omega_h}$ with vertices $t_h := (L/2, h/2 - h^4)$, $l_h := (L/2 - h^4, h/2)$, $r_h := (L/2 + h^4, h/2)$, and define
\begin{equation}
w_h(x) := 
\begin{cases}
0 &x_1 < L/2, x \notin \conv(t_h,r_h,l_h)\\
\frac1h(x-a_h)^\perp&x_1 > L/2, x\notin \conv(r_h,l_h,t_h)\\
\frac1h(v^\perp\otimes v)(x-a_h) &x\in \conv(r_h,l_h,t_h),
\end{cases}
\end{equation}
where $v = (e_1+e_2)/\sqrt{2} = (r_h - t_h)/|r_h - t_h|$. See Figure \ref{fig: incomplete} for a sketch of the corresponding deformation.
 
This displacement field jumps on the line segment $\gamma_h := \{L/2\}\times (-h/2, h/2 - h^4)$ and has elastic energy
\begin{equation}
\int_{\Omega_h \setminus \gamma_h} |ew_h|^2\,dx = \int_{\conv(t_h,r_h,l_h)} |\nabla w_h|^2\,dx = h^8/h^2 = h^6 \ll h^3.
\end{equation}

So even though $\Hm^1(J_{w_h}) = h-h^4$, there are no constant $\overline{A}_h, \overline{b}_h$ such that $w_h - \overline{A}_h x - \overline{b}_h$ is bounded on most of $\Omega_h$. To achieve convergence in measure, we have to allow $\overline{A}_h,\overline{b}_h$ to jump at $L/2$. 
\end{rem}

\begin{figure}
\includegraphics{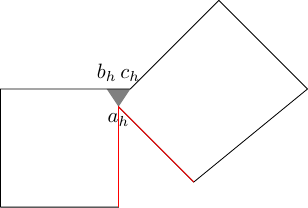}
\caption{Sketch of the deformation $x+w_h(x)$ in Remark \ref{rem: incomplete}. The length of the jump is slightly less than the height. Elastic stress, while high, is contained to the small triangle $\conv(a_h,b_h,c_h)$.}\label{fig: incomplete}
\end{figure}

\begin{rem}
\JG{Even for $\bar{A}_h=0, \bar{b}_h=0$} the compactness result only guarantees convergence in measure, but not in $L^1(\Omega_1;\R^2)$. As a result, we have that the minimizers of $E_h + F$ converge in measure to minimizers of $E_0 + F$ whenever $F$ is continuous under convergence in measure. Nontrivial linear functionals $F$ are of course not continuous under convergence in measure. Consider for example $F(y) := \int_{\Omega_1} y_2 \,dx$, and
\begin{equation}
y_h(x) := \begin{cases}
-\frac{1}{h^5} e_2 &, \text{ if }x\in B((L/2,0),h^2)\\
0&, \text{ otherwise.}
\end{cases}
\end{equation} 

Then $E_h(y_h) \to 0$, $y_h \to 0$ in measure but not in $L^1$, and $F(y_h) \to -\infty$.
\end{rem}

Next, we state the main $\Gamma$-convergence result.

\begin{thm} \label{thm: Gamma limit}
The $\Gamma$-limit of $E_h$ as defined in \eqref{eq: def energy} with respect to the convergence in Theorem \ref{thm: compactness}  is $E_0$. More precisely, we have
\begin{enumerate}[label=(\roman*)]
\item [(i)] For any $y\in \calA$ there is a sequence $y_h\in \SBV^2(\Omega_1;\R^2)$ such that $y_h \to y$ in $L^\infty(\Omega_1;\R^2)$ and $\lim_{h\to 0} E_h(y_h) = E_0(y)$. \label{item: upper bound gamma}

\item [(ii)] Let $y\in \calA$  and $y_h: \Omega_1 \to \R^2$, $\omega_h\subset \Omega_1$, and $A_h:(0,L)\to \Skew(2)$, $b_h:(0,L)\to \R^2$ two piecewise constant functions, such that the conditions of Theorem \ref{thm: compactness} hold. Then $E_0(y) \leq \liminf_{h\to 0} E_h(y_h)$. \label{item: lower bound gamma}
\end{enumerate}
\end{thm}

\begin{rem}
\JG{Note that the conditions for the lower bound (ii) in Theorem \ref{thm: Gamma limit} are satisfied if $y_h \to y$ in $L^2$, c.f.~Theorem \ref{thm: compactness} below. }
\end{rem}

\begin{rem}
For an isotropic material it holds $\mathbb{C} F : F = 2\mu |F_{sym}|^2 + \lambda \operatorname{tr}(F)^2$ where the Lam\'e coefficients satisfy $\mu>0$ and $2\mu +  \lambda > 0$. 
A straightforward computation then shows for the bending constant from \eqref{eq: bending constant} that $a = 2 \mu + 2 \frac{\mu\lambda}{2\mu+\lambda}$.
\end{rem}

\section{The space $\GSBD^2$ and Korn's inequality}\label{sec: GSBD}

We use the space of Generalized functions of Special Bounded Deformation with integrability $2$, written $\GSBD^2(\Omega)$, as the effective domain of the energies $F_h$, \JG{c.f.~\cite{dal2013generalised}}. This space is the natural topological function space for the brittle Griffith fracture model. It is analogous to the $\GSBV^p$ spaces that are widely used in image segmentation, \JG{see \cite{AmbrosioFuscoPallara} for the definition and the properties of the spaces $\GSBV^p$}. In order to recall the definition of the space $\GSBD^p$, we first recall the definition of the jump set and the approximate symmetric gradient.

\begin{defi}
\begin{itemize}
\item [(i)] Let $\Omega \subset \R^d$ be open, $x\in \Omega$, and $v:\Omega \to \R^N$ be measurable. The \emph{approximate limit} of $v$ at $x$, if it exists, is defined as the measurable function $\aplim_x v :\R^d \to \R^N$ such that
\[
v\left(\frac{\cdot - x}{r} \right) \stackrel{r \to 0}{\rightarrow} \aplim v_x\text{ in measure on }B(0,1).
\]

Note that the function $\aplim_x v$ is positively $0$-homogeneous, i.e. $\aplim_x v(ty) = \aplim_x v(y)$ for all $y\in \R^d$, $t>0$.

\item [(ii)] The \emph{jump set} $J_v$ of a measurable function $v:\Omega \to \R^N$ is the set of all points $x\in \Omega$ where $\aplim_x u$ exists and is of the form
\[
\aplim_x v (y) = \begin{cases}
a&\text{, if }y\cdot \nu > 0\\
b&\text{, otherwise}
\end{cases}
\]
for some $\nu\in S^{d-1}$, $a,b\in \R^N$.
\item[(iii)] A function $v: \Omega \to \R^d$ is said to have an approximate symmetric gradient $ev(x) \in \Symm(d)$ at $x \in \Omega$ if it holds
\[
\aplim_x \frac{ \left( u(y) - u(x) - ev(x) (y-x) \right) \cdot (y-x) }{|x-y|^2} = 0.
\]
\end{itemize}
\end{defi}

Next we recall the definitions of the spaces $\BD(\Omega)$, $\SBD^p(\Omega)$ and $\GSBD(\Omega)$.

\begin{defi}
\begin{itemize}
\item [(i)]
Let $\Omega \subset \R^d$ be open. A vector field $v\in L^1(\Omega;\R^d)$ is said to be of bounded deformation, $\BD(\Omega)$, if
\begin{equation}
\sup \left\{\int_\Omega v \cdot (\nabla \cdot \phi) \,dx\,:,\phi \in C_c^1(\Omega; \Symm(d)) , \|\phi\|_\infty \leq 1 \right\} < \infty.
\end{equation}

\item [(ii)]
The vector field is said to be of special bounded deformation with integrability $p\in (1,\infty)$, $\SBD^p(\Omega)$, if there is a tensor field $ev\in L^p(\Omega;\Symm(d))$ such that
\begin{equation}
\int_\Omega -v \cdot (\nabla \cdot \phi)\,dx = \int_\Omega ev:\Phi\,dx + \int_{J_v} [v] \cdot \phi \nu \,d\Hm^{d-1}
\end{equation}
for all $\phi\in C_c^1(\Omega; \Symm(d))$ and $\Hm^{d-1}(J_v) < \infty$ , where $J_v$ is the jump set of $v$.

\item [(iii)]
A measurable vector field $v:\Omega \to \R^d$ is said to be a function of generalized special bounded deformation, $\GSBD(\Omega)$, if there exists a bounded Radon measure measure $\lambda \in \mathcal{M}_{+}(\Omega)$ such that for every $\xi\in \R^d$ and for $\Hm^{d-1}$-almost every $y\in \xi^\perp \subset \R^d$ the function $f_{\xi,y}:\Omega_{\xi,y} \to \R$, where $\Omega_{\xi,y} := \{t\in \R\,;\,y+t\xi\in \Omega\}$, defined as
\[
f_{\xi,y}(t) :=  u(y+t\xi) \cdot \xi,
\]
is in $\SBV_{loc}(\Omega_{\xi,y})$ and it holds for every Borel set $B \subset \Omega$
\[
\int_{\xi^{\perp}} |D f_{\xi,y}|(B_{\xi,y} \cap \Omega_{\xi,y}) + \mathcal{H}^0(B_{\xi,y} \cap J_{f_{\xi,y)}} \cap \{|[ f_{\xi,y}] \geq 1\}) \, d\mathcal{H}^{n-1}(y) \leq \lambda(B),
\]
where $B_{\xi,y} = \{ t \in \R: y+t\xi \in B\}$.

\end{itemize}
\end{defi}

We recall from \cite{dal2013generalised} that for $v \in \GSBD(\Omega)$ it can be shown  that the approximate symmetric gradient $ev \in L^1(\Omega;\R^{d \times d})$ exists $\mathcal{L}^d$-a.e. and the jump set $J_v$ is a countably $\mathcal{H}^{d-1}$-rectifiable set with measure theoretic normal $\nu_v$.
Moreover, it holds for $\xi \in S^{d-1}$ and $\mathcal{H}^{d-1}$-a.e.~$y \in \xi^{\perp}$
\[
Df_{\xi,y} = ev(y+t \xi) \xi \cdot \xi \, dt +  \sum_{t\,:\,y+t\xi\in J_v} [f_{y,\xi}]\delta_t.
\]
Eventually, we define the set $\GSBD^p(\Omega)$.
\begin{defi}
Let $p \in (1,\infty)$. We say that $v \in \GSBD^p(\Omega)$ if $v \in \GSBD(\Omega)$, $ev \in L^p(\Omega)$ and $\mathcal{H}^{d-1}(J_v) < \infty$.
\end{defi}
For fine properties of the functions ins $\GSBD^p$ we refer to \cite{dal2013generalised}.

We now state a strong version of Korn's inequality for functions in $\GSBD^2(\Omega)$ in any dimension, which is found in \cite{CaChSc20}, for an earlier two-dimensional version see also \cite{FriedrichSchmidt_KornIn2d} or \cite{CoChFr16}.

\begin{prop}\label{prop: Korn}
\JG{Let $\Omega \subset \R^d$ be open, bounded, connected with Lipschitz boundary. Then there is a constant $C(\Omega)>0$ such that for all $w\in \GSBD^2(\Omega)$ there is a function $\bar w \in W^{1,2}(\Omega;\R^d)$ and a set set of finite perimeter $\omega \subseteq \Omega$ such that $w = \bar w$ on $\Omega \setminus \omega$.
\begin{equation}
|\omega| + \Hm^{d-1}(\partial \omega) \leq C(\Omega) \Hm^{d-1}(J_w)
\end{equation}
and 
\begin{equation}
\int_{\Omega} |e \bar w|^2 \, dx \leq C(\Omega) \int_{\Omega} |ew|^2 \, dx.
\end{equation}
Moreover, there exists a matrix $A\in \Skew(d)$ and a vector $b\in \R^d$ such that
\begin{equation}
\int_{\Omega \setminus \omega} |\nabla w - A|^2 + |w - Ax - b|^2  \,dx \leq C(\Omega) \int_\Omega |ew|^2\,dx,
\end{equation}}
\end{prop}

We note that the estimate is useless if $\Hm^{d-1}(J_w)$ is too large with respect to $\Omega$, as then we can simply take $\omega = \Omega$.

We shall use this result to define good rectangles in $\Omega_h$, noting that we never use the extension $\bar w$, only the bounds on the bad set $\omega$. For the rest of the article (excluding the appendix) we work only in $d=2$.

\begin{defi} \label{def: good rectangles}
Let $h,\delta > 0$, and $w \in \GSBD^2(\Omega_h)$. We consider all the rectangles $Q_z := (z- h, z + h) \times (-h/2,h/2)$ with $z=h,2h,\ldots, (\lfloor L/h \rfloor - 1) h$. We call a rectangle $Q_z$ $\delta$-good if
\begin{equation}
\Hm^1(J_w \cap Q_z) \leq \delta h,
\end{equation}
and bad otherwise.
For a good rectangle we denote by $\omega_z \subseteq Q_z$ the exceptional set from Proposition \ref{prop: Korn}.
\end{defi}

\begin{rem}\label{rem: neighboring rectangles}
\Jedit{Again, let $h,\delta > 0$ and $w \in \GSBD^2(\Omega_h)$. }We note that if $\delta\leq \delta_0$ for some universal constant $\delta_0$, then we have on a good rectangle $Q_z$ by Proposition \ref{prop: Korn} that $|Q_z\setminus \omega_z| > 7|Q_z|/8$ and there is a unique pair $A_z\in \Skew(2), b_z\in \R^2$ defined as
\begin{equation}
(A_z,b_z) := \arg\min_{A,b} \int_{Q_z \setminus \omega_z} \frac1{h^2}|w(x) - Ax  -b|^2 + |\nabla w(x) - A|^2\,dx,
\end{equation}
where minimization runs over $\Skew(2) \times \R^2$. Proposition \ref{prop: Korn} yields that
\begin{equation}
\int_{Q_z \setminus \omega_z} \frac1{h^2}|w(x) - A_z x -b_z|^2 + |\nabla w(x) - A_z|^2\,dx \leq C \int_{Q_z} |ew|^2\,dx.
\end{equation}

Then we see that for two neighboring $\delta$-good rectangles $Q_z, Q_{z+h}$ we have
\begin{equation}\label{eq: good neighbors} |A_z - A_{z+h}|^2 +  |b_z - b_{z+h}|^2 \leq C h^{-2} \int_{Q_z \cup Q_{z+h}} |e w|^2,
\end{equation}
for some universal constant $C$.
\end{rem} 

We now show a stronger version of \eqref{eq: good neighbors} for two good rectangles that are separated by a sequence of bad rectangles as long as there is not enough jump to separate the two rectangles completely:

\begin{prop}\label{prop: bridge}
Let $\eta\in (0,1)$. Then there is a constant $\delta(\eta)>0$ such that for all $N\in \N$ there is a constant $C(\eta,N)>0$ such that the following holds:

Let $h>0$, \Jedit{$w_h \in \GSBD^p(\Omega_h)$}, $Q_z$ and $Q_{z'}$ be two $\delta(\eta)$-good rectangles with $|z-z'|\leq Nh$, and
\begin{equation}
\Hm^1(J_{w_h} \cap \conv(Q_z \cup Q_{z'})) \leq (1-\eta) h.
\end{equation}

Then
\begin{equation}\label{eq: good bridge}
|A_z - A_{z'}|^2 +   |b_z - b_{z'}|^2 \leq C(\eta, N) h^{-2} \int_{\conv(Q_z \cup Q_{z'})} |e w_h|^2\,dx
\end{equation}
\Jedit{Here, the matrices $A_z,A_{z'} \in \Skew(2)$ and vectors $b_z,b_{z'} \in \R^2$ are the matrices and vectors given by Proposition \ref{prop: Korn} on the squares $Q_z$ and $Q_{z'}$, respectively. }
\end{prop}

\begin{figure}
\includegraphics[scale=1]{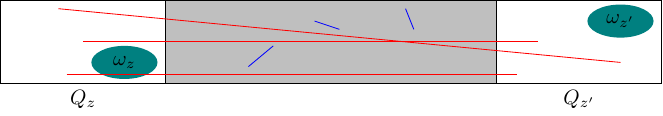}
\caption{
Two good rectangles $Q_z$, $Q_{z'}$ are separated by a series of bad rectangles, but without enough jump set to completely separate the two. Then we can find three line segments connecting the good rectangles that prevent the two from being infinitesimally rotated or shifted against one another.}\label{fig: bridge}
\end{figure}

\begin{proof}
  We assume without loss of generality that $h=1$ (by rescaling) and that $A_{z'},b_{z'} = 0$. This is achieved by replacing $w_h$ with $w_h(x) - A_{z'}x - b_{z'}$, which has the same jump set and elastic strain. We will write $w=w_h$ since $h=1$. We also assume that $z'> z$.

We show that there are three pairs $(p_i,q_i)\in Q_z \times Q_{z'}$, $i=1,2,3$, such that
\begin{equation}\label{eq: rigid motion bound}
\left| \left(A_z p_i + b_z \right)\cdot (p_i- q_i) \right|^2 \leq C(\eta, N) \int_{\conv(Q_z \cup Q_{z'})} |ew|^2 \,dx.
\end{equation}

If we can make sure that the three pairs are in general position, Lemmas \ref{lemma: line determinant} and \ref{lemma: determinant 2d} yield the upper bounds on $|A_z|, |b_z|$. See Figure \ref{fig: bridge} for a visual sketch. We first show how to construct the first two pairs, where $(p_i - q_i) \| e_1$ (the notation $v \| w$ means that the vectors $v,w\in \R^2$ are linearly dependent):

Start with the set of horizontal lines that do not cross $J_{w}$ in the sense of slicing for $\GSBD$ functions, or more precisely
\begin{equation}
H_1 := \{x_2 \in (-1/2,1/2)\,:\, x_1 \mapsto e_1\cdot w(x_1,x_2)\textrm{ is absolutely continuous on }(z - 1, z' + 1)\}.
\end{equation}

We have by the segment regularity in $\GSBD^2$ (see \cite{dal2013generalised}) that $\Hm^1(H_1) \geq \eta$.

We intersect $H_1$ with three large subsets of $(-1/2,1/2)$, namely
\begin{equation}
\begin{aligned}
H_2 := &\bigg\{x_2 \in (-1/2,1/2)\,:\, && \int_{z - 1}^{z'+1} |ew(x_1,x_2)|\,dx_1 \leq \frac8{\eta} \int_{\conv(Q_z \cup Q_{z'})} |ew(x)|\,dx\bigg\},\\
H_3 := &\bigg\{x_2 \in (-1/2,1/2)\,:\, &&\int_0^L (\mathds{1}_{\omega_z} (x_1,x_2) + \mathds{1}_{\omega_{z'}} (x_1,x_2))\,dx_1 \leq \frac8{\eta}(|\omega_z| + |\omega_{z'}|)\bigg\},\\
H_4 := &\bigg\{x_2 \in (-1/2,1/2)\,:\, &&\int_{0}^{L} \mathds{1}_{Q_z \setminus \omega_z}(x_1,x_2)|w(x_1,x_2) - A_z(x_1,x_2)^T - b_z|\,dx_1\\
&&& +  \int_{0}^{L}\mathds{1}_{Q_{z'} \setminus \omega_{z'}}(x_1,x_2)|w(x_1,x_2)|\,dx_1\\
&&& \leq \frac8{\eta}\left( \int_{Q_z \setminus \omega_z} |w(x) - A_zx - b|\,dx + \int_{Q_{z'} \setminus \omega_{z'}} |w(x)|\,dx \right)  \bigg\}.
\end{aligned}
\end{equation}

By Fubini's theorem and Markov's inequality we have $\Hm^1(H_2),\Hm^1(H_3),\Hm^1(H_4) \geq 1- \frac{\eta}{8}$, and thus
\begin{equation}
\Hm^1(H_1 \cap H_2 \cap H_3 \cap H_4) \geq \eta - \frac{3\eta}{8} > \frac{\eta}{2}.
\end{equation}

By Proposition \ref{prop: Korn} we have
\begin{equation}
\frac{8}{\eta}(|\omega_z| + |\omega_{z'}|) \leq \frac{16}{\eta} C \delta.
\end{equation}
If $\delta < \frac{\eta}{16C}$, this ensures that whenever $t\in H_3 \cap H_4$, there are points $x\in Q_z\setminus \omega_z$, $y\in Q_{z'} \setminus \omega_{z'}$ with $x_2 = y_2 = t$, $|x - y| \geq 1$, and
\begin{equation}\label{eq: Markov choice}
|w(x) - A_z x -b_z|  + |w(y)| \leq \frac{8}{\eta}\left( \int_{Q_z \setminus \omega_z} |w(x) - A_zx - b|\,dx + \int_{Q_{z'} \setminus \omega_{z'}} |w(x)|\,dx \right) 
\end{equation}

This allows us to choose the first two pairs $(p_i,q_i)\in (Q_{z}\setminus \omega_z) \times (Q_{z'} \setminus \omega_{z'})$, $i =1,2$ such that $(p_i)_2 = (q_i)_2\in H_1\cap H_2\cap H_3 \cap H_4$,
\begin{equation}
|p_i - q_i| \geq 1, \quad |(p_1)_2 - (p_2)_2| \geq \frac\eta{2}.
\end{equation}

By the definitions of $H_1,H_2,H_3,H_4$ we then have
\begin{equation}
\begin{aligned}
&|(p_i-q_i) \cdot (A_z p_i +b_z)|\\
 \leq & |(p_i-q_i) \cdot (w(p_i) - w(q_i))| + |(p_i-q_i) \cdot (w(p_i) - A_z p_i - b_z)| + |(p_i-q_i) \cdot w(q_i)|\\
\leq & C(\eta,N) \sqrt{\int_{\conv(Q_z \cup Q_{z'})} |ew(x)|^2\,dx},
\end{aligned}
\end{equation}
where in order to estimate the first term we used the fundamental theorem of calculus
\begin{equation}
(p_i-q_i) \cdot (w(p_i) - w(q_i)) = \int_{[p_i,q_i]} (p_i - q_i) \cdot \nabla w\frac{p_i -q_i}{|p_i - q_i|}\,d\Hm^1 = \int_{[p_i,q_i]} (p_i - q_i) \cdot ew \frac{p_i -q_i}{|p_i - q_i|}\,d\Hm^1.
\end{equation}

The fact that we may do so for almost every segment not intersecting the jump set is proved for $\GSBD$ functions in e.g. \cite{dal2013generalised}.

We now repeat the above argument to obtain one more pair $(p_3,q_3)$ with $p_3\in Q_z\setminus \omega_z$, $q_3\in Q_{z'}\setminus \omega_{z'}$. Instead of a horizontal line segment, we choose $(p_3 - q_3 ) \| e_\theta := \frac{1}{\sqrt{1+\theta^2}}(1,\theta)$ with $\theta := \frac{\eta^2}{N+2}>0$. We define analogously to before

\begin{equation}
\begin{aligned}
D_1 := &\bigg\{t \in (-1/2,1/2)\,:\,  s \mapsto e_\theta \cdot w(z+te_\theta^\perp + se_\theta)\textrm{ is absolutely continuous,}\\
&\quad \textrm{where }z+te_\theta^\perp + se_\theta\in \conv(Q_z \cup Q_{z'})\bigg\},\\
D_2 := &\bigg\{t \in (-1/2,1/2)\,:\,  \int_{\R} \mathds{1}_{\conv(Q_z \cup Q_{z'})}(z+te_\theta^\perp + se_\theta) |ew(z + te_\theta^\perp + se_\theta)|\,ds\\
& \quad \leq  \frac{8}{\eta} \int_{\conv(Q_z \cup Q_{z'})} |ew(x)|\,dx\bigg\},\\
D_3 := &\bigg\{t \in (-1/2,1/2)\,:\, \int_{\R} \mathds{1}_{\omega_z} (z+te_\theta^\perp + se_\theta) + \mathds{1}_{\omega_{z'}} (z+te_\theta^\perp + se_\theta)\,ds \leq \frac8{\eta}(|\omega_z| + |\omega_{z'}|)\bigg\},\\
D_4 := &\bigg\{t \in (-1/2,1/2)\,:\, 
  \int_{\R}\mathds{1}_{Q_{z'} \setminus \omega_{z'}}(z+te_\theta^\perp + se_\theta)|w(z+te_\theta^\perp + se_\theta)|\,ds +  \\
&\quad + \int_{\R} \mathds{1}_{Q_z \setminus \omega_z}(z+te_\theta^\perp + se_\theta)|w(z+te_\theta^\perp + se_\theta) - A_z(z+te_\theta^\perp + se_\theta)^T - b_z|\,ds \\ 
& \quad \leq \frac8{\eta}\left( \int_{Q_z \setminus \omega_z} |w(x) - A_zx - b|\,dx + \int_{Q_{z'} \setminus \omega_{z'}} |w(x)|\,dx \right)  \bigg\}.
\end{aligned}
\end{equation}

As before, we have $\Hm^1(D_1) \geq \eta$, $\Hm^1(D_2), \Hm^1(D_3), \Hm^1(D_4) \geq 1- \frac{\eta}{8}$, so that
\begin{equation}
\Hm^1(D_1 \cap D_2 \cap D_3 \cap D_4) \geq \frac{\eta}{2}.
\end{equation}
This allows us to pick a diagonal line $z+t+\R e_\theta$ with $t\in D_1 \cap D_2 \cap D_3 \cap D_4$ and if $\delta < \frac{\eta}{32C}$, we find a pair $p_3\in Q_z\setminus \omega_z$, $q_3\in Q_{z'}\setminus \omega_{z'}$ on the diagonal line with $|p_3 - q_3| \geq 1$ and such that \eqref{eq: Markov choice} holds. Note that as long as $|t \pm 1/2| > \frac{\eta}4$, the diagonal line intersects both $Q_z$ and $Q_{z'}$. By the definitions of $D_1, D_2, D_3, D_4$, \eqref{eq: rigid motion bound} holds also for $(p_3,q_3)$.

Define the linear map $F\in \mathrm{Lin}(\Skew(2)\times \R^2; \R^3)$ by
\begin{equation}
F(A,b) := ((Ap_i + b) \cdot (p_i-q_i))_{i=1,2,3}
\end{equation}
By \eqref{eq: rigid motion bound} we have
\begin{equation}
|F(A_z,b_z)| \leq C(\eta,N) \sqrt{\int_{\conv(Q_z \cup Q_{z'})} |ew|^2\,dx    }.
\end{equation}
Using the identity $F^{-1} = (\det F)^{-1} (\cof F)^T$, we have 
\[
|A_z| + |b_z| \leq |\det F|^{-1} |F|^2 |F(A_z,b_z)|.
\]
We clearly have $|F| \leq (N+3)^3$. By the direct calulations in Lemmas \ref{lemma: line determinant} and \ref{lemma: determinant 2d}, we may estimate
\begin{equation}
|\det F| \geq |p_1 - q_1||p_2 - q_2||p_3 - q_3|\, \frac{\eta}{2}|\sin\theta|^2 \geq c(\eta,N),
\end{equation}
since all three pairs have distance at least $1$, the parallel lines have distance at least $\eta/2$, and the angle of the diagonal line is $\theta$. This shows that
\begin{equation}
|A_z| + |b_z| \leq C(\eta,N) \sqrt{\int_{\conv(Q_z \cup Q_{z'})} |ew|^2\,dx    }
\end{equation}
whenever $\delta< \frac{\eta}{32C}$, completing the proof.

\end{proof}

\begin{rem}
We note here that the above procedure may be generalized to $d=3$ using Lemma \ref{lemma: determinant 3d} and employing $6$ segments instead of $3$. We also note that we only used the bound on $|\omega_z|$, not the one on its perimeter. In theory, it is also possible to employ a similar construction in nonlinear elasticity, using a geometric rigidity estimate as in \cite{FJM02} for $\GSBV^2$ functions with a small jump set.
\end{rem}

\begin{rem}
The choice of line segments connecting two elastic bodies in order to prevent independent rigid motions of either body is encountered in civil engineering in the context of trusses. The proof above, and its three-dimensional version, show that many such trusses exist. Lemmas \ref{lemma: determinant 2d} and \ref{lemma: determinant 3d} are potentially useful to the engineering community in the construction of optimal trusses in bridges, scaffolding, towers etc.
\end{rem}

\section{Proof of compactness}

Here we combine estimates \eqref{eq: good neighbors} and \eqref{eq: good bridge} to show that $\nabla w_h$ is very close in $L^2(\Omega_h;\R^{2\times 2})$ to some $A_h\in \SBV^2((0,L);\Skew(2))$ and $\#J_{A_h} \leq \left\lfloor \liminf_{h\to 0} \frac1h \Hm^1(J_{w_h}) \right\rfloor$.
An additional Poincar\'e inequality then yields estimates for $w_h$.

This will imply the compactness result and be useful for the proof of the lower bound.

\begin{prop}\label{prop: better compactness}
Let $w_h\in \GSBD^2(\Omega_h)$ with $\sup_{h>0} F_h(w_h) < \infty$. Let
\begin{equation}
M:= \left\lfloor \liminf_{h\to 0} \frac1h \Hm^1(J_{w_h}) \right\rfloor.
\end{equation}

Then there is a subsequence (not relabeled), a sequence of sets $\omega_{h} \subset \Omega_{h}$ and sequences $A_{h} \in \SBV^2((0,L);\Skew(2))$, $b_{h}\in \SBV^2((0,L);\R^2)$ such that
\begin{enumerate}[label=(\roman*)]
\item  $\frac1h |\omega_{h}| \to 0$ and $\frac1h \Hm^1(\partial^*\omega_{h}) \leq C$. \label{item: bounds omegah}
\item  $\#(J_{A_{h}} \cup J_{b_{h}}) \leq M$ and
\begin{equation}
\sup_{h > 0}\int_0^L |A_{h}'(x_1)|^2 + |b_{h}'(x_1)|^2\,dx_1 < \infty.
\end{equation} \label{item: bounds Ah bh}
\item 
\begin{equation}
\sup_{h > 0}  \frac1{h^3}\int_{\Omega_{h} \setminus \omega_{h}}|\nabla w_{h}(x) - A_{h}(x_1)|^2 \,dx  < \infty.
\end{equation} \label{item: bound nable wh Ah}
\item \label{item: bound wh - rigid motions} 
\begin{equation} 
\sup_{h > 0} \frac1{h^3} \int_{\Omega_{h} \setminus \omega_{h}} |w_h(x) - x_2A_{h}(x_1)e_2 - b_h(x_1)|^2\,dx < \infty.
\end{equation}
\end{enumerate}
\end{prop}

\begin{proof}
First, we may assume that $\liminf_{h \to 0} \frac1h \Hm^1(J_{w_h}) = \lim_{h\to 0} \frac1h \Hm^1(J_{w_h})$.
Now, let $h >0$ so small that $\frac1h \Hm^1(J_{w_h}) < M+1$.
This implies for all $\eta > 0$ small enough that it holds 
\[
\frac1{(1-\eta)h} \Hm^1(J_{w_h}) < M+1.
\]
In addition, let $\delta = \delta(\eta) > 0$ be as in Proposition \ref{prop: bridge}.
 
For $z \in \{h,2h, \dots, (\lfloor L/h \rfloor-1) h\}$, we write $Q_z = (z-h,z+h) \times (-h/2,h/2)$. 
We recall from Definition \ref{def: good rectangles} that $Q_z$ is called a $\delta$-good rectangle if $\Hm^1(J_{w_h} \cap Q_z) \leq \delta h$.
We write
\begin{equation}
\mathcal{G}_h = \{Q_z: z\in \{h,2h,\dots, (\lfloor L/h \rfloor-1) h\}, Q_z \text{ is a $\delta$-good rectangle} \}.
\end{equation}
Then we apply Proposition \ref{prop: Korn} to each $Q_z \in \mathcal{G}_h$ and obtain matrices $A_z \in \Skew(2)$, vectors $b_z \in \R^2$ and sets of finite perimeter $\omega_z \subseteq Q_z$ such that
\begin{equation} \label{eq: Korn and Korn Poincare rectangle}
\int_{Q_z \setminus \omega_z} |\nabla w_h - A_z|^2 + h^{-2} |w_h - A_z (x - (z,0)) - b_z|^2  \,dx \leq C \int_\Omega |e w_h|^2\,dx,
\end{equation}
and
\begin{equation}
|\omega_z| + \Hm^{1}(\partial^* \omega_z) \leq C \Hm^{1}(J_{w_h} \cap Q_z).
\end{equation}
Let us recall from Proposition \ref{prop: bridge}, c.f.~also Remark \ref{rem: neighboring rectangles}, that for the choice of $\delta$ and two neighboring rectangles $Q_z,Q_{z+h} \in \mathcal{G}_h$ it holds
\begin{equation} \label{eq: neighboring good rectangles}
h^2 |A_z - A_{z+h}|^2 +  h^2|b_z - b_{z+h}|^2 \leq C \int_{Q_z \cup Q_{z+h}} |e w_h|^2.
\end{equation}
Moreover, we write 
\begin{equation}
\mathcal{B}_h = \{Q_z: z\in \{h,2h,\dots, (\lfloor L/h \rfloor-1) h\}, Q_z \text{ is a $\delta$-bad rectangle} \}.
\end{equation}
As $E_h(w_h) \leq C$, we obtain that $\#\mathcal{B}_h \leq \frac{C}{\delta}$. 
Let us now denote by $(C_h^k)_{k=1}^{K_h}$ the connected components of $\bigcup_{Q_z \in \mathcal{B}_h} Q_z$, where $K_h \leq \frac{C}{\delta}$.
We define the connected components with a large jump set as
\begin{equation}
\mathcal{J}_h = \left\{ C_h^k: \frac1h \Hm^1(J_{w_h} \cap C_h^k) \geq 1 - \eta\right\}.
\end{equation}
By the choice of $\eta$ we find that $\# \mathcal{J}_h < M+1$ which implies that $\# \mathcal{J}_h \leq M$. 
Moreover, by Proposition \ref{prop: bridge} we find for our choice of $\delta$ and $C_h^k \notin \mathcal{J}_h$ that 
\begin{equation}\label{eq: bridge bad rectangles}
h^2 |A_z - A_{z'}|^2 +  h^2 |b_z - b_{z'}|^2 \leq C(\eta, N) \int_{\conv(Q_z \cup Q_{z'})} |e w_h|^2\,dx,
\end{equation}
where $Q_z, Q_{z'} \in \mathcal{G}_h$ are the good rectangles neighboring $C_h^k$ and $N = |z-z'| / h \leq \#\mathcal{B}_h \leq \frac{C}{\delta}$.

Now we can construct the functions $A_h \in \SBV((0,L);\Skew(2))$ by linearly interpolating the values $A_z$ between neighboring good rectangles and over connected components of bad rectangles which do not carry a lot of jump set, see Figure \ref{fig: interpolation}.  
Precisely, we define
\begin{equation}
A_h(t) = \begin{cases}
\frac{z+h - t}{h} A_z + \frac{t - z}{h} A_{z+h} &\text{ if } t \in (z,z+h) \text{ and } Q_z,Q_{z+h} \in \mathcal{G}_h, \\
\frac{z'-t}{z'-z} A_z + \frac{t-z}{z'-z} A_{z'} &\text{ if }  C_h^k = (z,z') \times (-h/2,h/2) \notin \mathcal{J}_h, \\ 
&\text{ and } t \in (z,z'), \\
\1_{(z, (z+ z')/2)}(t) A_z + \1_{((z+ z')/2, z')}(t) A_{z'} &\text{ if }  C_h^k = (z,z') \times (-h/2,h/2) \in \mathcal{J}_h, \\ 
&\text{ and } t \in (z,z').
\end{cases}
\end{equation}
If not already defined, we extend this definition constantly onto the intervals $(0,h)$ and $((\lfloor L/h \rfloor-1) h,L)$.

\begin{figure}
\begin{subfigure}{1\textwidth}\centering

\includegraphics[scale=1.3]{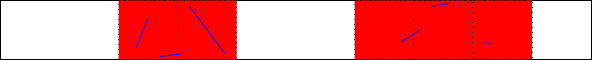}\subcaption{Sketch of the situation in $\Omega_h$. The bad rectangles are sketched in red, the jump set of $w_h$ is sketched in blue. The size of the jump set of $w_h$ in the left connected component of the union of bad rectangles contains more than $1-\eta$, whereas in the right connected component it is less than $1-\eta$.}
\end{subfigure}
\begin{subfigure}{1\textwidth}\centering

\includegraphics[scale=1.3]{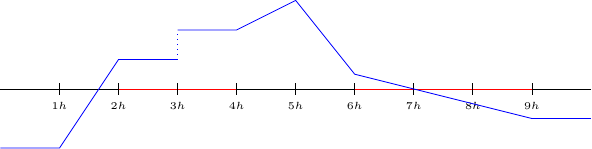} 
\subcaption{Sketch of the interpolation procedure to construct $A_h$. Between neighboring good cubes with center $kh$ and $(k+1)$ we interpolate linearly. In connected components of the union of the bad rectangles in which the jump set of $w_h$ is larger than $1-\eta$ the function $A_h$ jumps (left). In connected components in which the jump set of $w_h$ is less than $1-\eta$ we interpolate using Proposition \ref{prop: bridge}.}
\end{subfigure}
\caption{Interpolation procedure to construct $A_h$.}
\label{fig: interpolation}
\end{figure}

We define the function $b_h$ by interpolating in a similar fashion between the values $b_z$.

First we notice that the functions $A_h$ and $b_h$ can only jump at the center of connected components in $\mathcal{J}_h$. Consequently, 
\begin{equation}
\#(J_{A_h} \cup J_{b_h}) \leq \#\mathcal{J}_h \leq M.
\end{equation}
Moreover, using \eqref{eq: neighboring good rectangles} and \eqref{eq: bridge bad rectangles} we obtain (note that $A_h'$ and $b_h'$ denote the absolutely continuous part of $D A_h$ and $D b_h$, respectively)
\begin{align}
\int_0^L |A_h'|^2 +  |b_h'|^2 \,dx_1 \leq C(\eta) h^{-3}\int_{\Omega_h} |e w_h|^2 \, dx \lesssim C(\eta) h^{-3} F_h(w_h) \lesssim C(\eta).
\end{align}
This shows \ref{item: bounds Ah bh}.

Next, we define the exceptional set $\omega_h$ as the union of the exceptional sets on the good rectangles, all bad rectangles and a boundary layer, i.e.
\begin{equation}
\omega_h = \bigcup_{Q_z \in \mathcal{G}_h} \omega_z \cup \bigcup_{Q_z \in \mathcal{B}_h} Q_z \cup ((\lfloor L/h \rfloor -1 )h, L) \times (-h/2,h/2).
\end{equation}
It follows from the properties of $\omega_z \subseteq Q_z$ that
\begin{equation}
\Hm^1(\partial^* \bigcup_{Q_z \in \mathcal{G}_h} \omega_z) \leq \sum_{Q_z \in \mathcal{G}_h} \Hm^1(\partial^* \omega_z) \leq C\sum_{Q_z \in \mathcal{G}_h} \Hm^1(Q_z \cap J_{w_h}) \leq 2C \Hm^1(J_{w_h}). 
\end{equation}
By the the subadditivity of the squareroot and the isoperimetric inequality it follows that
\begin{align}
\mathcal{L}^2\left( \bigcup_{Q_z \in \mathcal{G}_h} \omega_z \right)^{\frac12} \leq \sum_{Q_z \in \mathcal{G}_h} \mathcal{L}^2(\omega_z)^{\frac12} \leq C \sum_{Q_z \in \mathcal{G}_h} \Hm^1(\partial^* \omega_z) \leq 2 C \Hm^1(J_{w_h}).  
\end{align}
Hence, it follows \Jedit{(recall $\# \mathcal{B}_h \leq \frac{C}{\delta}$)}
\begin{align}
&\Hm^1(\partial^* \omega_h) \leq C \Hm^1(J_{w_h}) + 2 \#\mathcal{B}_h h + h \leq C(1+1/\delta) h, \\
 &\mathcal{L}^2(\omega_h) \leq C \Hm^1(J_{w_h})^2 + \# \mathcal{B}_h h^2 + h^2 \leq C(1+1/\delta) h^2,
\end{align} 
which shows \ref{item: bounds omegah}.

In order to prove \ref{item: bound nable wh Ah} we estimate using \eqref{eq: Korn and Korn Poincare rectangle} and H\"older's inequality 
\begin{align}
&\int_{\Omega_h \setminus \omega_h} |\nabla w_h(x) - A_h(x_1)|^2 \, dx \\
\leq &2 \sum_{Q_z \in \mathcal{G}_h} \int_{Q_z \setminus \omega_z} |\nabla w_h(x) - A_z|^2 + |A_z - A_h(x_1)|^2 \, dx \\
\leq &C \sum_{Q_z \in \mathcal{G}_h} \int_{Q_z} |ew_h|^2 \, dx + \sum_{Q_z \in \mathcal{G}_h} 2h^3 \int_{z-h}^{z+h} |A_h'(x_1)|^2 \, dx_1 \\
\leq &2C \int_{\Omega_h} |ew_h|^2 \, dx + 4  h^3 \int_0^L |A_h'(x_1)|^2 \, dx_1 \\
 \leq &C h^3. \label{eq: estimate nala wh - Ah}
\end{align}
Note that for the last inequality we used \ref{item: bounds Ah bh}.

It remains to show \ref{item: bound wh - rigid motions}.
We recall that by Proposition \ref{prop: Korn} we have for all $Q_z \in \mathcal{G}_h$ that
\begin{equation}
\int_{Q_z \setminus \omega_z} |w_h(x) - A_zx - b_z|^2 \, dx \Jedit{\leq C h^2 \int_{Q_z} |ew|^2\,dx}.
\end{equation}
Hence, we obtain from the definition of $\omega_h$ and similar estimates as in \eqref{eq: estimate nala wh - Ah} that
\begin{align}
&\int_{\Omega_h \setminus \omega_h} |w_h(x) - A_h(x_1) x - b_h(x_1)|^2 \, dx \\
\leq &\sum_{Q_z \in \mathcal{G}_h} \int_{Q_z \setminus \omega_z} |w_h(x) - A_h(x_1) x - b_h(x_1)|^2 \, dx \\
\leq &C \sum_{Q_z \in \mathcal{G}_h} \int_{Q_z \setminus \omega_z} |w_h(x) - A_z x - b_z|^2 + |A_z - A_h(x_1)|^2 + |b_z - b_h(x_1)|^2 \, dx \\
\leq &C \sum_{Q_z \in \mathcal{G}_h} h^2 \int_{Q_z} |e w_h|^2 \, dx + \sum_{Q_z \in \mathcal{G}_h} 2h^3 \int_{z-h}^{z+h} |A_h'(x_1)|^2 + |b_h'(x_1)|^2 \, dx_1 \\
\leq & C h^2 F_h(w_h) + C h^3 \int_0^L  |A_h'(x_1)|^2 + |b_h'(x_1)|^2 \, dx_1 \\ 
\leq &C h^3.
\end{align}

\end{proof}

\begin{proof}[Proof of Theorem \ref{thm: compactness}]
Let us define $T_h: \Omega_h \to \Omega_1$ by $T_h(x_1,x_2) = (x_1,x_2/h)$. 
Then we write $\sigma_h = T_h(\omega_h)$, where $\omega_h$ is the set constructed in Proposition \ref{prop: better compactness}. 
From Proposition \ref{prop: better compactness} \ref{item: bounds omegah} we immediately obtain \ref{item: bounds baromegah}.

Next, we write 
\begin{equation}
M = \left\lfloor  \liminf_{h\to 0}\int_{J_{y_h}} (\nu_1,\frac1h \nu_2) \, d\Hm^1  \right\rfloor.
\end{equation}
Let $A_h$ and $b_h$ be the functions from Proposition \ref{prop: better compactness}. 
Recall from \ref{item: bounds Ah bh} in Proposition \ref{prop: better compactness} that $A_h \in \SBV((0,L);\Skew(2))$ and $b_h \in \SBV((0,L);\R^2)$ with
\begin{equation} \label{eq: estimate Ah and bh}
\#(J_{A_h} \cup J_{b_h}) \leq M \text{ and } \sup_h \int_0^L |A_{h}'(x_1)|^2 + |b_{h}'(x_1)|^2\,dx_1 < \infty.
\end{equation}
We write $J_{A_h}\cup\{0,L\} = \{t_1, \dots, t_{N_h}\}$ and define the function $\bar{A}_h: (0,L) \to \Skew(2)$ by (see also Figure \ref{fig: piecewise constant})
\begin{equation}
\bar{A}_h(t) = \fint_{t_i}^{t_{i+1}} A_h(s) \, ds \text{ if } t \in (t_i,t_{i+1}).
\end{equation}
It follows immediately that $\bar{A}_h$ is piecewise constant and $J_{\bar{A}_h} \subseteq J_{A_h}$. 
By a minor modification of $\bar{A}_h$, if necessary, we may assume that $J_{\bar{A}_h} = J_{A_h}$. 

\begin{figure}
\centering
\includegraphics[scale=1.3]{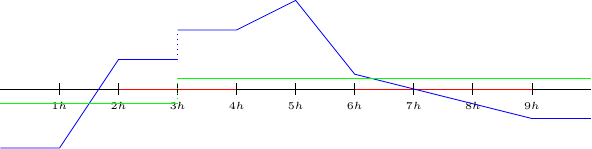}
\caption{Sketch of the construction of $\bar{A}_h$. The function $A_h$ evolves from piecewise interpolation. Its graph (more precisely, the graph of the upper right entry of the matrix field $A_h$) is sketched in blue. The corresponding graph of $\bar{A}_h$ (i.e. the upper right entry of the matrix field $\bar{A}_h$) is sketched in green. The field $\bar{A}_h$ is constant between jump points of $A_h$.}
\label{fig: piecewise constant}
\end{figure}

We deduce from \eqref{eq: estimate Ah and bh} that $\| A_h - \bar{A}_h \|_{L^{\infty}(\Omega_1;\R^{2\times 2})} \leq C$ and that $A_h - \bar{A}_h$ is a bounded sequence in $\SBV^2((0,L);\Skew(2))$.
By the compactness properties of the space $\SBV^2$, there exists a (not relabeled) subsequence and $A \in \SBV^2((0,L);\Skew(2))$ such that $A_h - \bar{A}_h \to A$ in $L^1((0,L);\R^{2\times 2})$, $A_h' \rightharpoonup A'$ in $L^2((0,L);\R^{2\times 2})$ and $\#J_{A} \leq \liminf_h \# J_{A_h}$.

The definition of $\bar{b}_h$ is completely analogous. 
Again, we may assume without loss of generality that $J_{\bar{b}_h} = J_{b_h}$ and obtain that $\| b_h - \bar{b}_h\|_{L^{\infty}((0,L);\R^2)} \leq C$. Then there exists a (not relabeled) subsequence and $b \in \SBV^2((0,L);\R^2)$ such that $b_h - \bar{b}_h \to b$ in $L^1((0,L);\R^2)$, $b_h' \rightharpoonup b'$ in $L^2((0,L);\R^2)$, and $\# J_{b} \leq \liminf_{h \to 0} \# J_{b_h}$.

Let us now define the functions $u_h: \Omega_h \to \R^2$ by
\begin{equation}
u_h(x) = (A_h(x_1) - \bar{A}_h(x_1) )x + (b_h(x_1) - \bar{b}_h(x_1)).
\end{equation}
It follows that $u_h \in \SBV^2(\Omega_h;\R^2)$.
Moreover, we obtain from the bounds and properties of $A_h - \bar{A}_h$ and $b_h - \bar{b}_h$ that
\begin{equation}\label{eq: properties uh}
\| u_h \|_{L^{\infty}(\Omega_h;\R^2)} \leq C, \int_{\Omega_h} |\nabla u_h|^2 \, dx \leq C h \text{ and } J_{u_h} \subseteq  (J_{A_h} \cup J_{b_h}) \times (-h/2,h/2).
\end{equation}
Accordingly, we define the rescaled functions $\bar{u}_h \in \SBV^2(\Omega_1; \R^2)$ by
\begin{equation}\label{eq: def baruh}
\bar{u}_h(x_1,x_2) = u_h(x_1,hx_2) = (A_h(x_1) - \bar{A}_h(x_1) ) (x_1,hx_2)^T + (b_h(x_1) - \bar{b}_h(x_1)),
\end{equation} 
which satisfies by \eqref{eq: properties uh} 
\begin{equation}\label{eq: properties baruh}
\| \bar{u}_h \|_{L^{\infty}(\Omega_1;\R^2)} \leq C, \int_{\Omega_1} |\nabla \bar{u}_h|^2 \, dx \leq C \text{ and } J_{\bar{u}_h} \subseteq  (J_{A_h} \cup J_{b_h}) \times (-1/2,1/2).
\end{equation}
In particular, $\bar{u}_h$ is a bounded sequence in $\SBV^2(\Omega_1;\R^2)$. 
Moreover, by the definition, \eqref{eq: def baruh}, and the compactness properties of $A_h - \bar{A}_h$ and $b_h - \bar{b}_h$ it follows $\bar{u}_h \to y$ in $L^1(\Omega_1;\R^2)$, where
\begin{equation} \label{eq: specific form y}
y(x_1,x_2) = A(x_1)(x_1,0)^T + b(x_1).
\end{equation}
In particular, $y$ does not depend on $x_2$.
In order to show $y \in \mathcal{A}$ we still need to prove that $\partial_1 y_1 = 0$ and $\partial_1 y_2 \in \SBV(\Omega_1)$.

Next, let us define $\tilde{w}_h: \Omega_h \to \R^2$ by $\tilde{w}_h(x) = \left(w_h(x) - \bar{A}_h(x_1)x - \bar{b}_h(x_1)\right) \1_{\Omega_h \setminus \omega_h}$. 
From Proposition \ref{prop: better compactness} \ref{item: bounds omegah} and \ref{item: bound nable wh Ah} we deduce that $\tilde{w}_h \in \GSBV^2(\Omega_h;\R^2)$ and
\begin{align}
\int_{\Omega_h} |\nabla \tilde{w}_h|^2 \, dx &= \int_{\Omega_h \setminus \omega_h} |\nabla w_h - \bar{A}_h|^2 \, dx \\
&\leq 2 \int_{\Omega_h \setminus \omega_h} |\nabla w_h - A_h|^2 + |A_h - \bar{A}_h|^2 \, dx \\
&\leq C(h^3 + h).
\end{align}
Moreover, $\Hm^1(J_{\tilde{w}_h}) \leq \Hm^1(\partial^* \omega_h) + \Hm^1(J_{w_h}) + h \#(J_{A_h} \cup J_{b_h}) \leq Ch$. 
It follows from Proposition \ref{prop: better compactness} \ref{item: bound wh - rigid motions} that
\begin{equation}
\int_{\Omega_h} |\tilde{w}_h|^2 \, dx \leq C h.
\end{equation}
The above shows that the functions $\tilde{y}_h: \Omega_1 \to \R^2$ defined by 
\[
\tilde{y}_h(x_1,x_2) = \left(y_h(x_1,x_2) - \bar{A}_h(x_1)  (x_1,hx_2) - \bar{b}_h(x_1)\right) \1_{\Omega_1 \setminus \sigma_h}
\]
form a bounded sequence in $L^2(\Omega_1;\R^2)$, are in $\GSBV^2(\Omega_1;\R^2)$ and satisfy
\begin{equation}\label{eq: bounds tildeyh}
\int_{\Omega_1} | \nabla_h \tilde{y}_h|^2 \, dx \leq C \text{ and } \int_{J_{\tilde{y}_h}} |(\nu_1, \frac1h \nu_2)| \, d\Hm^1 \leq C,
\end{equation}
where $\nu \in S^1$ is the measure theoretic normal to $J_{\tilde{y}_h}$.
It follows immediately that $\tilde{y}_h$ is a bounded sequence in $\GSBV^2(\Omega_1;\R^2)$.

Moreover, it follows from Proposition \ref{prop: better compactness} \ref{item: bound wh - rigid motions} that
\begin{align}
\tilde{y}_h - \bar{u}_h \1_{\Omega_1 \setminus \sigma_h} \to 0 \text{ in } L^2(\Omega_1;\R^2).
\end{align}
As $\mathcal{L}^2(\sigma_h) \to 0$ (c.f.~Proposition \ref{prop: better compactness} \ref{item: bounds omegah}), it follows that $\tilde{y}_h \to y$ in $L^2(\Omega_1;\R^2)$, which is \ref{item: convergence y_h}.
By the compactness properties of $\GSBV^2(\Omega_1;\R^2)$ it even follows that 
\begin{equation}
\nabla \tilde{y}_h \rightharpoonup \nabla y \text{ in } L^2(\Omega_1;\R^{2\times 2}).
\end{equation}
On the other hand, we find that
\begin{equation}\label{eq: conv nablah yh}
(\partial_1 \tilde{y}_h, \frac1h \partial_2 \tilde{y}_h) = \left(\nabla_h y_h - A_h\right) \1_{\Omega_1 \setminus \bar{\omega}_1} + \left(A_h - \bar{A}_h\right) \1_{\Omega_1 \setminus \bar{\omega}_1} \rightharpoonup A \text{ in } L^2(\Omega_1;\R^{2\times 2}).
\end{equation}
However, this implies that $A = (\partial_1 y, c)$ for a function $c \in \SBV^2(\Omega_1;\R^2)$. 
As $A$ is skew-symmetric, it follows immediately that $\partial_1 y_1 = 0$. 
Moreover, we deduce from the fact that $A \in \SBV^2((0,L);\Skew(2))$ and that $\partial_1 y \in \SBV^2(\Omega_1)$.
This concludes the proof of $y \in \mathcal{A}$.

It remains to show the inequalities in \ref{item: inequality jumpsets}.
First, let us recall that by construction we have $J_{\bar{A}_h} \cup J_{\bar{b}_h} = J_{A_h} \cup J_{b_h}$.
Consequently, it follows from Proposition \ref{prop: better compactness} \ref{item: bounds Ah bh} and the usual change of coordinates that 
\begin{equation}\label{eq: bound JbarAh}
\#(J_{\bar{A}_h} \cup J_{\bar{b}_h}) \leq \left\lfloor\liminf_{h\to 0} \int_{J_{y_h}} |(\nu_1,\frac1h \nu_2)| \,d\Hm^1 \right\rfloor,
\end{equation}
which is the second inequality in \ref{item: inequality jumpsets}.
In addition, we note that $J_{\partial_1 y} \subseteq J_A \times (-1/2,1/2)$.
Moreover, by the specific form of $y$, \eqref{eq: specific form y}, we have $J_y \subseteq J_{A} \cup J_{b}$.
However, lower semicontinuity yields $\#(J_A \cup J_b) \leq \liminf_h \#( J_{A_h - \bar{A}_h} \cup J_{b_h - \bar{b}_h}) \leq \liminf_h \#(J_{\bar{A}_h} \cup J_{\bar{b}_h})$.
Consequently,
\begin{equation}
\Hm^1(J_y \cup J_{\partial_1 y}) \leq \liminf_h \#(J_{A_h} \cup J_{b_h}),
\end{equation}
which yields using \eqref{eq: bound JbarAh} the remaining first inequality of \ref{item: inequality jumpsets}.
\end{proof}

\section{The lower bound}

\Jedit{In this section, we will prove the lower bound claimed in Theorem \ref{thm: Gamma limit}. 
As usual we may consider sequences $(y_h)_h$ which are equibounded in energy. In particular, due to the compactness theorem, up to a subsequence this sequence has a limit $y$ in the sense of Theorem \ref{thm: compactness}.}
First we prove the following result to identify the limit of the rescaled gradient of $y_h$.

\begin{prop} \label{prop: identification limit}
Define the function $W_h: \Omega_1 \to \R$ as
\begin{equation}
W_h(x_1,x_2) := \frac1h \mathds{1}_{\Omega_1 \setminus \sigma_h} \partial_1 (y_h)_1(x_1,x_2) ,
\end{equation}
where $\sigma_h = \{(x_1,x_2) \in \Omega_1: (x_1,hx_2) \in \omega_h \}$ for $\omega_h \subseteq \Omega_h$ from Proposition \ref{prop: better compactness}.

Then, in the setting of Theorem \ref{thm: compactness}, every subsequence of $W_h$ has a subsequence converging weakly in $L^2(\Omega_1)$ to a function $W \in L^2(\Omega_1)$ of the form 
\[
W(x_1,x_2)= -x_2 \, \partial_1 \partial_{1} y_2(x_1) + T(x_1),
\]
where $T \in L^2((0,L))$ and $y\in \mathcal{A}$ is any function for which the conclusion of Theorem \ref{thm: compactness} hold.
\end{prop}

\begin{rem}
Simlarly, one could identify the limit of $\frac1h \mathds{1}_{\Omega_1 \setminus \sigma_h} (\partial_1 y_h(x_1,x_2) - A_h e_1)$ as a function of the form $-x_2 \, A'(x_1) e_2 + T(x_1)$, where $A_h$ are the skew-symmetric functions constructed in Proposition \ref{prop: better compactness} and $A$ is the corresponding limit constructed in the proof of Theorem \ref{thm: compactness}.
As $A$ is skew-symmetric it holds $A'(x_1) e_2 \cdot e_2 = 0$. 
Hence, it is sufficient to determine the first component in this setting.
In the derivation of rod theories, this first entry only carries the information on bending but not on torsion, see \cite{GiGl21}.
\end{rem}

\begin{proof}
Recall that $w_h: \Omega_h \to \R^2$ is defined by $w_h(x_1,x_2) = y_h(x_1, x_2 /h)$ and $\sigma_h = \{(x_1,x_2) \in \Omega_1: (x_1,hx_2) \in \omega_h \} \subseteq \Omega_1$ is the analogue of $\omega_h \subseteq \Omega_h$ from Proposition \ref{prop: better compactness}.
Moreover, let $A_h \in \SBV^2(\Omega_h;\Skew(2))$ be the fields from Proposition \ref{prop: better compactness}.
By point \ref{item: bound nable wh Ah} of Proposition \ref{prop: better compactness}, we have  
\begin{equation}
\int_{\Omega_h \setminus \omega_h} \frac1{h^2}|\nabla w_h(x) - A_h(x_1)|^2  \, dx \leq Ch. 
\end{equation}
In particular, we have that $\int_{\Omega_h \setminus \omega_h} \frac1{h^2}|\partial_1 (w_h)_1(x)|^2  \, dx \leq Ch$.\footnote{$(w_h)_1$ denotes the first component of the vector $w_h\in \R^2$.}
A change of variables implies immediately that $W_h$ is bounded in $L^2(\Omega_1;\R^{2})$, so that every subsequence has a weakly convergent subsequence. 

Let us fix such a limit $W \in L^2(\Omega_1;\R^2)$ and a corresponding (not relabeled) subsequence.

Before we can prove the specific form of the limit $W$, let us recall from the proof of Theorem \ref{thm: compactness} that we have by construction for the sequences $A_h$ from Proposition \ref{prop: better compactness} and the sequences $\bar{A}_h$ from the proof of Theorem \ref{thm: compactness} that $A_h - \bar{A}_h \stackrel{*}{\rightharpoonup} A$ in $\SBV^2((0,L);\Skew(2))$ and that it holds $A = (\partial_1 y,c)$, where $c = \begin{pmatrix} -\partial_1 y_2 \\ 0\end{pmatrix}$ since $A$ is skew-symmetric. 
In particular, it holds that $(A_h^{1,2})' \rightharpoonup  -\partial_1 \partial_1 y_2$ in $L^2((0,L))$ since $\bar{A}_h$ is piecewise constant.
We note that for any other limit $\tilde{y}$, sequences of functions $\bar{A}_h$ and $\bar{b}_h$, and sets $\sigma_h$ for which the conclusions of Theorem \ref{thm: compactness} hold it can be shown that it still holds $\partial_1 \tilde{y} = \partial_1 y + d$, where $d \in \SBV((0,L);\R^2)$ is piecewise constant. 
In particular, $\partial_1 \partial_1 \tilde{y} = \partial_1 \partial_1 y$.
This shows that we can restrict ourselves to the specific $y \in \mathcal{A}$ that we construct in the proof of Theorem \ref{thm: compactness}.

To show the form of $W$, we consider for fixed $z\in (0,1/2)$ the second difference $V_h:\Omega \cap (\Omega-s_h e_1) \times (-1/2,1/2 - z) \to \R^2$, depending on $s_h>0$, defined by
\begin{align}
V_h(x) :=   \frac{\mathds{1}_h(x)\mathds{1}_{\Omega_1 \setminus \sigma_h}(x)}{h s_h} \bigg( &(y_h)_1(x + z e_2 + s_h e_1) - (y_h)_1(x + z e_2) \\& - (y_h)_1(x + s_h e_1) + (y_h)_1(x) \bigg).
\end{align}

Here, $\mathds{1}_h(x)= 1$ whenever $y_h$, $A_h$, and $A$ are absolutely continuous along the boundary of the rectangle spanned by the four points $x$, $x+ze_2$, $x+ze_2+s_h e_1$, $x+ s_h e_1$ and the boundary of the rectangle does not intersect $\partial \sigma_h$, and $0$ otherwise. 

As the jump sets of $A$ and $A_h$ a purely horizontal, it holds
\begin{align}\label{eq: char convergence}
\int_{\Omega_1} 1 -\mathds{1}_h(x) \,dx \leq &2|s_h| (\Hm^1(J_{y_h}) + \Hm^1(J_A) + \Hm^1(\partial^{*} \sigma_h \cap \Omega_1) + \#J_{A_h}) \\ &+ 2 z \left(\int_{J_{y_h}} |\nu \cdot e_2|\,d\Hm^1 + \int_{\partial^{*} \sigma_h} |\nu \cdot e_2|\,d\Hm^1 \right),
\end{align}
Here, $\nu$ denotes the measure theoretic normal to $J_{y_h}$ and $\partial^* \sigma_h$, respectively.  
By Theorem \ref{thm: compactness} \ref{item: bounds baromegah} and the bounds on the energy the right hand side tends to $0$ as $h\to0$ as long as $s_h \to 0$.

Now let $f\in C_c^\infty(\Omega \times (-1/2,1/2-z))$ be a test function. Then we have for $s_h \to 0$ with $\psi_h(x,t) = \frac1h \partial_1 (y_h)_1(x+ze_2 + ts_he_1) - \frac1h \partial_1 (y_h)_1(x+ts_he_1)$
\begin{align*}
&\int_{\Omega_1} V_h(x) f(x)\,dx\\
 = & \int_{\Omega_1} \left( \int_0^1 \psi_h(x,t) \,dt \right) \mathds{1}_h(x)\mathds{1}_{\Omega_1 \setminus \sigma_h}(x)f(x)\,dx \\
 = & \int_{\Omega_1} \left( \int_0^1 W_h(x+ze_2 + ts_he_1) - W_h(x+ts_he_1) \,dt \right) \mathds{1}_h(x) f(x)\,dx \\
\stackrel{h \to 0}{\longrightarrow} & \int_{\Omega_1} (W(x+ze_2)- W(x)) f(x)\,dx.
\end{align*}

Conversely, we have for $s_h = h^{\frac12}$ with $\varphi_h(x,t) = \frac{1}{h s_h} \partial_2 (y_h)_1(x+tze_2 + s_h e_1) - \frac{1}{h s_h} \partial_2 (y_h)_1(x+tze_2)$
\begin{align*}
&\lim_{h\to 0}\int_{\Omega_1} V_h(x) f(x)\,dx\\
 = & z \lim_{h\to 0}\int_{\Omega_1} \left( \int_0^1 \varphi_h(x,t) \,dt \right) \mathds{1}_h(x)\mathds{1}_{\Omega_1 \setminus \sigma_h}(x)f(x)\,dx \\
 = & z \lim_{h\to 0} \int_{\Omega_1} \frac1{s_h}\left(  A^{1,2}_h(x_1+s_h e_1)  - A^{1,2}_h(x_1)  \right) \mathds{1}_h(x) \mathds{1}_{\Omega_1 \setminus E_h}(x)f(x)\,dx \\
 = & z \lim_{h\to 0} \int_{\Omega_1} \left( \int_0^1  (A^{1,2}_h)'(x_1+ t s_h e_1)   \,dt \right) \mathds{1}_h(x) \mathds{1}_{\Omega_1 \setminus E_h}(x)f(x)\,dx \\
= & z \int_{\Omega_1} (A^{1,2})'(x) \, f(x)\,dx,
\end{align*}
For the second equality, we used the independence of $A_h, A$ from $x_2$ and the bound $\|(\frac1h \partial_2 y_h - A_h e_2)\mathds{1}_{\Omega_1 \setminus \sigma_h}\|_{L^2(\Omega_1;\R^2)} \leq Ch$, c.f.~\ref{item: bound nable wh Ah} in Proposition \ref{prop: better compactness}.

It follows that $W(x+ze_2)  - W(x)  = z \, (A^{1,2})'(x_1)$ for almost every $x\in \Omega_1$. 
Now, recall from above that
\begin{align}
(A^{1,2})'(x) = \partial_1 (-  \partial_1 y_2(x_1) ) = -  \partial_1 \partial_1 y_2(x_1).
\end{align}  
Consequently we obtain $W(x_1,x_2+z) - W(x_1,x_2) = - z \, ( \partial_1 \partial_1 y_2(x_1) )$ for almost every $x_1 \in (0,L)$. We define $T(x_1) := \int_{-1/2}^{1/2} W(x_1,z)\,dz$.
It follows for almost every $x_1 \in (0,L)$
\begin{align}
W(x_1,x_2) - T(x_1) & = \int_{-1/2}^{1/2} W(x_1,x_2) - W(x_1,z) \, dz \\
& = \int_{-1/2}^{1/2} -(x_2-z) \, (\partial_1 \partial_1 y_2(x_1) ) \, dz \\
&=  -x_2 \,  (\partial_1 \partial_1 y_2(x_1)).
\end{align}
\end{proof}

Now we can prove the lower bound in Theorem \ref{thm: Gamma limit}.
\begin{proof}[Proof of Theorem \ref{thm: Gamma limit} (i)]
By standard arguments we may always assume that it holds $\liminf_{h\to 0} E_h(y_h) = \lim_{h \to 0} E_h(y_h)$ and  $\sup_h E_h(y_h) < \infty$.

Then we notice that by the assumed convergence it holds that, c.f. Theorem \ref{thm: compactness} \ref{item: inequality jumpsets},
\begin{equation} \label{eq: lowerbound jumps}
\Hm^1( J_y \cup J_{\partial_1 y}) \leq \liminf_{h\to 0} \int_{J_{y_h}} |(\nu_1,\frac1h \nu_2)| \, d \Hm^1.
\end{equation}

Next, we show that
\[
\liminf_{h \to 0} \frac12 \int_{\Omega_1} \mathbb{C} \nabla_h y_h : \nabla_h y_h \, dx \geq \frac1{24} \int_{\Omega_1} a |\partial_1 \partial_1 y_2|^2 \, dx,
\]
where $a>0$ is defined in \eqref{eq: bending constant}.
First, we apply Proposition \ref{prop: better compactness} to the function $w_h(x_1,x_2) = y_h(x_1,x_2/h)$ to obtain sets $\omega_h \subseteq \Omega_h$ and $A_h \in \SBV((0,L);\Skew(2))$ for which the properties \ref{item: bounds omegah}, \ref{item: bounds Ah bh} and \ref{item: bound nable wh Ah}  of Proposition \ref{prop: better compactness} hold.
Again we denote by $\sigma_h$ the analogue of $\omega_h \subseteq \Omega_h$ in $\Omega_1$.
Then (up to a subsequence) it holds by Proposition \ref{prop: identification limit} that the sequence $W_h(x_1,x_2) = \frac1h \1_{\Omega_1 \setminus \sigma_h} \, ( \partial_1 (y_h)_1(x_1,x_2) )$ 
converges weakly in $L^2(\Omega_1)$ to a function $W \in L^2(\Omega_1)$ of the form
\begin{equation}
W(x_1,x_2) = -x_2 \, (\partial_1 \partial_1 y_2(x_1) ) + T(x_1),
\end{equation}
where $T \in L^2((0,L))$.
It follows by the definition of the bending constant $a$, c.f. \eqref{eq: bending constant}, that
\begin{align} \label{eq: lsc bulk}
\liminf_{h\to 0} \frac1{h^2} \frac12 \int_{\Omega_1} \mathbb{C} \nabla_h y_h : \nabla_h y_h \, dx \geq&\liminf_{h\to 0} \frac12 \int_{\Omega_1} a \,|W_h(x)|^2 \, dx \\ 
\geq & \frac12 \int_{\Omega_1} a \, |W(x)|^2 \, dx 
\end{align} 
Using the specific form of $W$ we compute
\begin{align}
\int_{\Omega_1} a\, |W |^2 \,dx = &\int_{\Omega_1} a \, |-x_2 \partial_1 \partial_1 y(x_1) + T(x_1)|^2 dx \\
\geq &\int_0^L \int_{-1/2}^{1/2} \left(a \, x_2^2 \left(\partial_1 \partial_1 y_2(x_1) \right)^2 - 2a x_2  \,(\partial_1 \partial_1 y_2(x_1)) \, T(x_1) \right)\, dx_2 dx_1 \\
=& \frac1{12} \int_0^L a \, |\partial_1 \partial_1 y_2(x_1) |^2 \, dx_1. \label{eq: lower bound elastic}
\end{align}
Combining \eqref{eq: lowerbound jumps} and \eqref{eq: lower bound elastic} yields
\begin{equation}
\liminf_{h \to 0} E_h(y_h) \geq \frac1{24}\int_0^L a  |\partial_1 \partial_1 y_2(x_1) |^2 \, dx_1 + \beta \Hm^1(J_y \cup J_{\partial_1} y).
\end{equation}

\end{proof}

\section{The upper bound}\label{sec: upper}

In this section we prove the existence of a recovery sequence for the energy $E_0$, i.e. we show Theorem \ref{thm: Gamma limit} (ii).

\begin{proof}[Proof of Theorem \ref{thm: Gamma limit} (ii)]
We may assume that $y \in \mathcal{A}$ (otherwise there is nothing to show).
Hence, $y \in \SBV^2(\Omega_1; \R^2)$ does not depend on $x_2$ , $\partial_1 y_2 \in \SBV^2(\Omega_1)$ and $\partial_1 y_1 = 0$.

In particular, we remark that $\partial_1 \partial_1 y_2 \in L^2(\Omega_1)$ does not depend on $x_2$.
Therefore we can find for every $\eta >0$ a function $g_{\eta} \in H^1(\Omega_1)$ which does not depend on $x_2$ such that $\| -\partial_1 \partial_1 y_2 - g_{\eta} \|_{L^2(\Omega_1)} \leq \eta$.
By the embedding $H^1((0,L)) \hookrightarrow L^{\infty}((0,L))$ in one dimension it follows that $\| g_{\eta} \|_{L^{\infty}(\Omega_1)} \leq C(\eta)$.
In addition, we note that since $\partial_1 y_2 \in \SBV^2(\Omega_1)$ does not depend on $x_2$ we obtain that $\partial_1 y_2 \in L^{\infty}(\Omega_1)$. Since $\partial_2 y_2 = 0$, we obtain $\nabla y_2 \in L^{\infty}(\Omega_1;\R^2)$.

Moreover, we recall from the definition of the bending coefficient $a$, c.f.~\eqref{eq: bending constant}, that there exist $b,c \in \R$ such that
\begin{equation}
a = \mathbb{C} \begin{pmatrix}
1 && b \\ 0 && c
\end{pmatrix} : \begin{pmatrix}
1 && b \\ 0 && c
\end{pmatrix}. 
\end{equation}
 
Now we define $y_h: \Omega_1 \to \R^2$ by
\begin{equation}
y_h(x_1,x_2) = y(x_1,0) - x_2 h \nabla y_2(x_1,0) + \frac12 x_2^2 h^2 g_{\eta}(x_1,0) \begin{pmatrix} b \\ c \end{pmatrix}.
\end{equation} 
One sees immediately that $y_h \to y$ in $L^{\infty}(\Omega_1;\R^2)$. \\

Moreover, the jumps of $y_h$ only occur in $x_1$-direction and $J_{y_h} \subseteq J_{y} \cup J_{\partial_1 y}$. In addition, $\nabla_h y_h = (\partial_1 y_h,\frac1h \partial_2 y_h)$ is given by 
\begin{align}
&\begin{pmatrix} - hx_2 \, \partial_1 \partial_1 y_2 + \frac12 x_2^2 h^2 \, b \,\partial_1 g_{\eta} & -\partial_1 y_2 + x_2 h \, b g_{\eta}\\
\partial_1 y_2  + \frac12 x_2^2 h^2 \, c \, \partial_1 g_{\eta} & x_2 h \,c g_{\eta} 
\end{pmatrix} \\ 
&= \underbrace{\begin{pmatrix} 0 & -\partial_1 y_2 \\ \partial_1 y_2 & 0 \end{pmatrix}}_{\text{skew-symmetric}} 
+ \begin{pmatrix}  - hx_2 \, \partial_1 \partial_1 y_2 + \frac12 x_2^2 h^2 \, b \, \partial_1 g_{\eta} &  x_2 h \, b g_{\eta} \\
  \frac12 x_2^2 h^2 \, c \,\partial_1 g_{\eta} & x_2 h \, c g_{\eta}  \end{pmatrix}.
\end{align}
Plugging into the elastic energy and using the $x_2$-independence of the occuring functions we find
\begin{align}
&\frac12 \int_{\Omega_1} \mathbb{C} (\partial_1 y_h,\frac1h \partial_2 y_h) : (\partial_1 y_h,\frac1h \partial_2 y_h) \, dx \\
=& \frac12 \int_{\Omega_1} h^2 x_2^2 \mathbb{C} \begin{pmatrix}  -\partial_1 \partial_1 y_2 &  b g_{\eta}\\
  0 &  b g_{\eta} 
  \end{pmatrix} : \begin{pmatrix}  -\partial_1 \partial_1 y_2 &  b g_{\eta}\\
  0 &  b g_{\eta} 
  \end{pmatrix} \\ &\quad +  x_2^3h^3 \, \mathbb{C} 
  \begin{pmatrix}  - \,\partial_1 \partial_1 y_2 &   b g_{\eta} \\
 0 &  c  g_{\eta} 
 \end{pmatrix} : \begin{pmatrix}   b (\partial_1 g_{\eta}) &  0\\
  c(\partial_1 g_{\eta})& 0  \end{pmatrix} \\
  &\quad + \frac14 x_2^4 h^4 (\partial_1 g_{\eta})^2 \, \mathbb{C} \begin{pmatrix}   b &  0\\
  c& 0  \end{pmatrix} :\begin{pmatrix}   b &  0\\
   c& 0  \end{pmatrix} \, dx  \\
  = & \frac12 \int_{\Omega_1} h^2 x_2^2 \mathbb{C} \begin{pmatrix}  -\partial_1 \partial_1 y_2 &   -(\partial_1 \partial_1 y_2) \, b \\
  0 &   -(\partial_1 \partial_1 y_2) \, c 
  \end{pmatrix} : \begin{pmatrix}  -\partial_1 \partial_1 y_2 &   -(\partial_1 \partial_1 y_2) \, b \\
  0 &   -(\partial_1 \partial_1 y_2) \, c 
  \end{pmatrix} \\  
  &\quad - h^2 x_2^2 \mathbb{C} \begin{pmatrix}  0 &   \partial_1 \partial_1 y_2 + g_{\eta} \, b \\
  0 &   \partial_1 \partial_1 y_2 + g_{\eta} \, c 
  \end{pmatrix} : \begin{pmatrix}  - 2\partial_1 \partial_1 y_2 &   -(\partial_1 \partial_1 y_2 - g_{\eta}) \, b \\
  0 &   -(\partial_1 \partial_1 y_2 - g_{\eta}) \, c 
  \end{pmatrix} \\
  &\quad + \frac14 x_2^4 h^4 (\partial_1 g_{\eta})^2 \, \mathbb{C} \begin{pmatrix}   b &  0\\
  c& 0  \end{pmatrix} :\begin{pmatrix}   b &  0\\
   c& 0  \end{pmatrix} \, dx \\ 
   \leq & \frac12 \int_{\Omega_1} h^2 x_2^2 a \, |\partial_1 \partial_1 y_2|^2 \, dx + C h^2 (\| \partial_1 \partial_1 y_2 \|_{L^2(\Omega_1)} + \eta) \| \partial_1 \partial_1 y_2 + g_{\eta} \|_{L^2(\Omega_1)} + C h^4 C(\eta) \\
   \leq & \frac12 \int_{\Omega_1} h^2 x_2^2 a \, |\partial_1 \partial_1 y_2|^2 \, dx + C \eta h^2 (\| \partial_1 \partial_1 y_2 \|_{L^2(\Omega_1)}+\eta) + C h^4 C(\eta).
\end{align} 
As $J_{y_h} \subseteq J_{y} \cup J_{\partial_1 y}$ is purely vertical, it follows that 
\begin{align}
&\limsup_{h\to 0} E_h(y_h)  \\
= &\limsup_{h\to 0} h^{-2}\int_{\Omega_1} \frac12 \mathbb{C}(\nabla_h y_h): (\nabla_h y_h) \, dx + \beta \int_{J_{y_h}} |(\nu_1,\nu_2/h)| \, d\Hm^1 \\
\leq &\frac12 \int_{\Omega_1} x_2^2 a \, |\partial_1 \partial_1 y_2|^2 \, dx + C\eta + \beta \Hm^1(J_{y} \cup J_{\partial_1 y}) \\
= &\frac1{24} \int_{\Omega_1} a \,| \partial_1\partial_1 y_2|^2 \, dx + \beta \Hm^1(J_{y} \cup J_{\partial_1 y}) + C\eta.
\end{align}
A diagonal argument in $\eta$ and $h$ finishes the proof.
\end{proof}

\section{Discussion}\label{sec: discussion}

In this section we want to briefly comment on the presented results and discuss future directions.
First, we acknowledge that the presented $\Gamma$-convergence result is very similar to the $n$-dimensional results already obtained in \cite{BaHe16} and \cite{AlTa20}.
However, we note that the compactness result presented in this paper and also the topology used for the $\Gamma$-convergence result is more general.
In \cite{BaHe16} uniform $L^{\infty}$-bounds are assumed a priori, in \cite{AlTa20} the authors invoke a compactness result in $\GSBD$ due to Chambolle and Crismale, \cite{ChCr19}.
The latter result essentially yields for a sequence $(y_h)_h \subseteq \GSBD(\Omega_1)$ such that $\sup_h \int_{\Omega_1} |e(w_h)|^2 \, dx < \infty$ and $\sup_h \mathcal{H}^{d-1}(J_{y_h}) < \infty$ that there is a (not relabeled) subsequence for which the set $A = \{ x \in \Omega_1: y_h(x) \to \infty\}$ has finite perimeter and outside $A$ one has $y_h \to y$ pointwise a.e.~for some $y \in \GSBD(\Omega_1)$. Moreover, $\mathcal{H}^{d-1}(\partial^* A \cup J_y) \leq \liminf_h \mathcal{H}^{d-1}(J_{y_h})$.
The compactness statement in this work also characterizes the limiting behavior of the sequence $y_h$ on the set $A$. Namely, we identify suitable rigid motion which we locally subtract from $y_h$ such that the remaining sequence is essentially compact in $L^2(\Omega_1;\R^2)$. In addition, these rigid motions do not create additional jump in the limit, c.f.~Theorem \ref{thm: compactness}.
Whether the used techniques can be generalized to higher dimensions to obtain an improved version of the result in \cite{ChCr19}, needs to be studied.

Moreover, we note that the presented analysis is closely related to the techniques presented in \cite{Sc17}, where the author derives a beam theory from a rotationally invariant model.
A key role in the analysis plays a geometric rigidity estimate for functions $y \in \SBV$, \cite{FriedrichSchmidt_KornIn2d}, which subdivides the domain into different regions in which the function $y$ is close to a rigid motion, see \cite[Theorem 3.5]{Sc17}. The estimate is essentially sharp in the sense that the perimeter of the identified regions is up to a small error controlled by the size of the jump set of $y$.

The presented techniques can be extended to derive a three-dimensional rod theory for brittle materials in the linearized setting. 
Similarly to two-dimensional beams, three-dimensional rods can undergo stretching and bending. However, in addition one can observe torsion, i.e. the twisting of the rod around its axis, see, for example, \cite{MoMu04} and the references therein.
In contrast to bending or stretching, in order to capture torsion in the limit, one has to keep track of the limit of $\nabla w_h \approx A_h$. 
Note that a three-dimensional version of Proposition \ref{prop: bridge} can be proved using Lemma \ref{lemma: determinant 3d} instead of Lemma \ref{lemma: determinant 2d}.
A corresponding paper is in preparation, \cite{GiGl21}.



\appendix
\section{Rigidity for trusses}

\begin{defi}\label{defi: linear map}
Let $d, N\in \N\setminus \{0\}$. Given $N$ pairs $(p_i,q_i) \in \R^d \times \R^d$, $i=1,\ldots,N$, consider the linear map $F: \Skew(d) \times \R^d \to \R^N$ defined by
\begin{equation}
F(A, b) = \left( (Ap_i + b)\cdot (p_i - q_i) \right)_{i=1}^N.
\end{equation}
\end{defi}

We are interested in the special case $N = \dim(\Skew(d)) + d = \frac{d(d+1)}{2}$, where $F\in \R^{N\times N}$ is a square matrix. We then ask for which pairs $(p_i,q_i)$ the associated $F$ is invertible, and what is its determinant.

\begin{rem}\label{rem: rays}
Assume that $q_i = 0$ for all $i=1,\ldots,N$. Then $F(A, b) = F(0,b)$ for all $A\in \Skew(d)$, since $Ax \cdot x = 0$ for all $x\in \R^d$. In particular, $F$ is not invertible.
\end{rem}

\begin{lemma}\label{lemma: line determinant}
Let $d\in \N\setminus\{0\}$, $N = \frac{d(d+1)}{2}$, $p_i \neq q_i$. Then there exists a function $f_d : (G^d)^N \to \R$ such that
\begin{equation} \label{eq: formula f}
\det F = d_1 \ldots d_N f_d(L_1,\ldots,L_N), 
\end{equation}
where $d_i = |p_i - q_i|$, and
\begin{equation}L_i= \left[\left(p_i, \frac{p_i - q_i}{|p_i - q_i|}\right)\right]\in (\R^d \times S^{d-1}) / \sim
\end{equation}
is the oriented line obtained by extending the line segment $[p_i,q_i]$, which is an element of the oriented affine line Grassmannian $G^d := \R^d \times S^{d-1} / \sim$, with $(x,v) \sim (z,w)$ whenever $v = w$ and $x-z \| \ v$.

In addition, the function $f_d$ has the property $f_d\circ \sigma = \det\sigma f_d$ for all permutations $\sigma$. Replacing any one line $[(p_i,v_i)]$ with its opposite $[(p_i,-v_i)]$ flips the sign of $f_d$ and $f_d$ is invariant under a single rotation or translation applied to all lines $L_1, \dots, L_N$.
\end{lemma}

\begin{proof}
Let $L_i = [(p_i,v_i)] \in G^d$. 
We define $f_d: (G^d)^N \to \R$ by
\[
f_d(L_1,\dots, L_N) = \det F,
\]
where $F: \Skew(d) \times \R^d \to \R^N$ is the linear mapping such that 
\[
F(A,b) = ((Ap_i + b) \cdot v_i)_{i=1}^N.
\]
This definition is independent from the representative.
Indeed, let $[(q_i,w_i)] = L_i$. 
Then $p_i-q_i$ is parallel to $v_i=w_i$.
Since $A$ is skew-symmetric, we find that 
\[
(Aq_i + b)\cdot w_i = (Ap_i + b) \cdot v_i + (A(q_i-p_i)) \cdot w_i = (Ap_i + b)\cdot v_i. 
\]
This shows that $f_d$ is well-defined. 
The formula \eqref{eq: formula f} and the properties of $f_d$ follow immediately from the properties of the determinant.

\end{proof}

\begin{rem}
By Remark \ref{rem: rays} and Lemma \ref{lemma: line determinant}, we have $\det F = 0$ if all extended lines intersect in the same point or at least two lines are the same. We now show that these are the only cases where this happens.
\end{rem}

\begin{lemma}\label{lemma: determinant 2d}
Let $L_1,L_2,L_3\in G^2$ be three oriented lines. Then
\begin{itemize}
\item [(i)] If $L_1,L_2, L_3$ are parallel, then $f_2(L_1,L_2,L_3) = 0$.

\item [(ii)] If $L_1$ intersects $L_2$ in $p\in \R^2$ with angle $\alpha \in (0,\pi/2)$ and $L_3$ in $q\in \R^2$ with angle $\beta \in (0,\pi/2)$, then 
\begin{equation}\label{eq: determinant 2d}
|f_2(L_1,L_2,L_3)| = |p-q| |\sin\alpha| |\sin\beta|.
\end{equation}
\end{itemize} 
\end{lemma}

Note that this covers all cases. Note that if $L_1,L_2,L_3$ are pairwise transversal, then formula \eqref{eq: determinant 2d} is independent of the order of $L_1,L_2,L_3$. If two of the three lines are parallel, we have to fix the order so that $L_1$ is the transversal direction. In this case $\alpha = \beta$ in \eqref{eq: determinant 2d}, and this is indeed the case in the proof of Proposition \ref{prop: bridge}.

\begin{proof}
We work with variables $a,b_1,b_2$, where $A = a\begin{pmatrix}0&1\\-1&0\end{pmatrix}$ and $b = (b_1,b_2)^T$ which defines the standard basis of $\Skew(2) \times \R^2$. In these coordinates, if $L_i = [(p_i, v_i)]$, the mapping from Definition \ref{defi: linear map} is represented by a matrix $\tilde{F} \in \R^{3\times 3}$ whose rows are $(p_i\wedge v_i, (v_i)_1, (v_i)_2)\in \R^3$, where as usual $x \wedge y = x_1 y_2 - x_2 y_1$ for $x,y \in \R^2$ . 

In case (i), $v_1 = v_2 = v_3$, the rows of $\tilde{F}$ are linearly dependent, and $\det \tilde{F} = 0$.
Hence, $f(L_1,L_2,L_3) = 0$.

In case (ii), we apply a rigid motion, so that $v_2 = (1,0)^T$, $p_1 = p_2 = p = (0,0)^T$. Then $p_3 = q = \pm |p-q|v_1$. We calculate explicitly
\begin{align*}
|f_2(L_1,L_2,L_3)| &= \left| \det \begin{pmatrix}0&(v_1)_1&(v_1)_2\\0&1&0\\(p-q)\wedge v_3&(v_3)_1&(v_3)_2\end{pmatrix} \right| \\
&= |(p-q)\wedge v_3||v_1 \wedge v_2| = |p-q||\sin\beta| |\sin\alpha| .
\end{align*}
\end{proof}

We now turn to the three-dimensional case. The space $\Skew(3) \times \R^3$ has the canonical basis
\begin{equation}
(\alpha, \beta, \gamma, b_1,b_2,b_3) \mapsto \left(\begin{pmatrix} 0 & \alpha & \beta\\-\alpha & 0 & \gamma\\ -\beta & -\gamma & 0\end{pmatrix}, \begin{pmatrix} b_1\\b_2\\b_3 \end{pmatrix} \right).
\end{equation}

If $L_i = [(p_i, v_i)]$, the $i$th row of $F$ is then given by
\begin{equation}
(p_i \times v_i, v_i)\in \R^6,
\end{equation}
where $p_i \times v_i$ denotes the usual cross product in $\R^3$.

\begin{lemma}\label{lemma: determinant 3d}
Let $L_i = [(p_i,v_i)]\in G^3$ for $i=1,\ldots,6$, and assume that $v_1 = v_2 = v_3$. Define $\overline{v} := v - (v_1\otimes v_1) v \in (v_1)^\perp$ for $v\in \R^3$ and $\overline{L}_i = [(\overline{x}_i, \frac{\overline{v}_i}{|\overline{v}_i|})]\in G^3$ for $i=4,5,6$ which we interpret as an element of $G^2$. Then
\begin{equation}\label{eq: determinant 3d}
|f_3(L_1,\ldots,L_6)| = |(\overline{x}_2-\overline{x}_1)\wedge (\overline{x}_3 - \overline{x}_1)| |\overline{v}_4||\overline{v}_5||\overline{v}_6| |f_2(\overline{L}_4, \overline{L}_5, \overline{L}_6)|.
\end{equation}
\end{lemma}

\begin{proof}
We apply a rigid motion so that $v_1 = (1,0,0)^T$, $p_1 = (0,0,0)^T$. Then we may interpret the vectors $\bar{v}_i$ and $\bar{x}_i$ to be elements of $\R^2$. In this way we can compute
\begin{equation}
\begin{aligned}
&|f_3(L_1,\ldots,L_6)| \\
= &\begin{vmatrix}
(0,0,0)^T & (p_2 \times v_1)^T  & (p_3 \times v_1)^T  & (p_4 \times v_4)^T  & (p_5 \times v_5)^T  & (p_6 \times v_6)^T  \\
v_1^T & v_1^T & v_1^T & v_4^T & v_5^T & v_6^T
\end{vmatrix}\\
=&
\begin{vmatrix}
(p_2 \times v_1)^T  & (p_3 \times v_1)^T  & (p_4 \times v_4)^T  & (p_5 \times v_5)^T  & (p_6 \times v_6)^T  \\
(0,0)^T & (0,0)^T & \overline{v}_4^T & \overline{v}_5^T & \overline{v}_6^T
\end{vmatrix}\\
=&
|\overline{x}_2 \wedge \overline{x}_3|\begin{vmatrix}
(\overline{x}_4 \wedge \overline{v}_4)^T  & (\overline{x}_5 \wedge \overline{v}_5)^T  & (\overline{x}_6 \wedge \overline{v}_6)^T  \\
\overline{v}_4^T & \overline{v}_5^T & \overline{v}_6^T
\end{vmatrix}\\
=&
|\overline{x}_2 \wedge \overline{x}_3|
|\overline{v}_4 | |\overline{v}_5||\overline{v}_6| |f_2(\overline{L}_4,\overline{L}_5, \overline{L}_6)|.
\end{aligned}
\end{equation}
\end{proof}

We see that we can guarantee that $f_3(L_1,\ldots,L_6)$ is bounded away from $0$ if the three parallel lines are chosen with a large cross-area, the other three lines are transversal enough and the intersection points of the projections of $L_4,L_5, L_6$ do not all coincide.

\bibliographystyle{abbrv}
\bibliography{2d_to_1d_revised}

\end{document}